\documentclass[12pt,a4paper]{article}
\usepackage[utf8]{inputenc}
\usepackage{amsmath}
\usepackage{amsfonts}
\usepackage{amssymb}
\usepackage{gensymb}
\usepackage{amsthm}
\usepackage{graphicx}
\usepackage{url}
\author{Kostyantyn Mazur}
\title{Convex Hull of $\left(t, t^2, \cdots, t^N\right)$}
\begin{document}
\newtheorem{lem}{Lemma}
\newtheorem*{thmb}{Theorem}
\newtheorem*{thme}{Theorem}
\newtheorem*{cor}{Corollary}
\newtheorem{defn}{Definition}
\theoremstyle{remark}
\newtheorem*{rmk}{Remark}

\maketitle

\begin{center}
\textbf{Abstract}
\end{center}
This paper analyzes the convex hull of the parametric curve $\left(t, t^2, \cdots, t^N\right)$, where $t$ is in a closed interval. It finds that every point in the convex hull is representable as a convex combination of $\frac{N+1}{2}$ points on the curve. It also finds that the evaluation of the convex combination is a homeomorphism from the convex combinations of $\frac{N+1}{2}$ points on the curve to the convex hull of the curve, as long as the points are listed in increasing order, and as long as two representations that are reachable from each other by removing terms with coefficient zero, combining terms with the same point, the inverses of these operations, or a sequence of these operations in any order, are considered to be equivalent.

\section{Introduction}
This paper is about finding a way to represent the points in the convex hull of the parametrically-defined curve $\left(t, t^2, t^3, \cdots, t^N\right)$ in $N$-dimensional real space, where $t$ is in some closed interval of real numbers. One way to do this is to use Carathéodory's Theorem to show that each point in the convex hull of this curve is a convex combination of at most $N+1$ points on the curve. That theorem, however, leaves many possible ways to name (as such a convex combination) most points in the convex hull, because the possible convex combinations of $N+1$ points on a $1$-dimensional object form a $\left(2N+1\right)$-dimensional set, so points will have an $N$-dimensional set of namings. 

This means that such namings have lots of redundant information, and make functions on this convex hull have lots of critical points when this convex hull is expressed in terms of these namings. This would be averted if there could be a way to use fewer than $N+1$ points on the curve to name each point in the convex hull. Ideally, each naming would use only $\frac{N+1}{2}$ points, as there, the namings form a $N$-dimensional set, which is the same dimension as that of the original object. That avoids having redudant information, but cannot be done blindly, because it is possible that some point in the convex hull does not have a naming. In other words, there might be a point in the convex hull that cannot be written as a convex combination of as few as $\frac{N+1}{2}$ points on the curve. The theorem that is the subject of this paper shows that that is not so; that, in fact, for any point in the convex hull, $\frac{N+1}{2}$ points on the curve always suffice (although both the points on the curve, and their coefficients in the convex combination, are allowed to vary when the point in the convex hull changes).

If $N$ is odd, then $\frac{N+1}{2}$ is an integer, so the optimal result of an $N$-parameter naming is achieved. If $N$ is even, then, still $\frac{N+1}{2}$ points on the curve are used, the half of a point being a point whose coefficient is allowed to vary, but which is fixed at the left endpoint on the interval. This point only contributes $1$ to the dimension of the set of namings, as opposed to the $2$ contributed by a point, for which both the position and coefficient are allowed to vary. This turns out to also be enough to name every point in the convex hull. In that sense, whether $N$ is even or odd, every point in the convex hull is a convex combination of $\frac{N+1}{2}$ points on the curve. This is the theorem proven in this paper. This paper also shows that if the convex combinations are written in order of increasing $t$, and if all ''equivalent'' convex combinations (in the sense that one can be reached from another by adding or removing terms with coefficient zero, and by combining or splitting terms with the same $t$) are considered to be the same one, then the function that actually evaluates the convex combination is a homeomorphism.

Other approaches to finding the convex hull of $\left(t, t^2, \cdots, t^N\right)$ are not ruled out, particularly those using the new field of convex algebraic geometry, which is introduced in \cite{Schweighofer-Sturmfels-Thomas}. For instance, Sinn's \cite{Sinn2015} would help in establishing the boundaries of this convex hull, if something known as the convex dual (defined in that paper) of this curve were known, and if this convex hull were to be shown to be a semi-algebraic set (that is, the solution set of a finite system of inequalities, each of which is of the form ``fixed polynomial in the $N$ variables is greater than $0$'' or ``fixed polynomial in the $N$ variables is not less than $0$'', or a finite union of such solution sets). Here, this (the convex dual) would be the set of vectors $v$ in $\Re^N$, such that $v \cdot \left(x, x^2, \cdots, x^N\right) \geq -1$, where $\cdot$ is the usual dot product. This (when the dot product is expanded) entails finding the whole set of polynomials of degree $N$ or less that have constant coefficient $1$ and that are nonnegative on the entirety of a closed interval. More would need to be done, such as finding the boundary of this set. It might be possible to find the convex hull this way.

Another approach is to follow Section 5.2 of Vinzant's \cite{Vinzant}, or to follow Scheiderer's \cite{Scheiderer-Semidefinite}, to get a semidefinite programming representation of the faces of $\left(t, t^2, \cdots, t^N\right)$ with $t_{min} \leq t \leq t_{max}$ (or any other truncated polynomial curve, or even a one-dimensional semi-algebraic subset of $\Re^N$ in the case of \cite{Scheiderer-Semidefinite}). Either result would give the convex hulls one at a time, which does not in itself give any general properties that hold for all $N$. In the context of \cite{Vinzant}, this paper provides a step towards an answer to its question about the general structure of the convex hull of a polynomial curve, all of whose components are monomials.

Scheiderer's result in \cite{SCHEIDERER20112606} shows the existence of lifted linear-matrix-inequality representations (a term defined in that paper) of convex hulls of a certain class of curves. Henrion's \cite{Henrion2011} shows how to rewrite the convex hull of a specific rational curve as a projection of a set of vectors, such that when the linear combination (using the components of a vector as coefficients) of specified matrices is taken, the result is a positive semidefinite matrix for and only for vectors in the set. Ranestad and Sturmfels' \cite{Ranestad-Sturmfels} studies the boundaries of a three-dimensional curve. Ranestad and Sturmfels' \cite{Ranestad2011} expresses the convex hull of an algebraic variety in terms of the secant planes of varying dimensions (with a $k$-dimensional secant plane intersecting the variety in at least $k+1$ points). Gouveia, Parrilo, and Thomas' \cite{doi:10.1137/090746525} approximates the convex hull of an algebraic variety by constructing a specific sequence of potentially larger objects than the convex hull, called theta bodies (a term defined in that paper). Sedykh's \cite{Sedykh1977} classifies the local properties of the boundaries of the convex hulls of almost all smooth curves that are functions from a circle to $\Re^3$. Unfortunately, none of these results apply here, at least not directly, because, while $\left(t, t^2, \cdots, t^N\right)$ is an algebraic curve (and an algebraic variety), its portion when $t$ is restricted to a closed interval is not (or, at least, $\left(t, t^2, \cdots, t^N\right)$ with $t_{min} \leq t \leq t_{max}$ is not a representation of the curve as an algebraic curve). In the case of \cite{Sedykh1977}, again the presence of boundaries makes this curve (with $N = 3$) not a smooth function from a circle to $\Re^3$.

One reason to study the convex hull of a truncated polynomial curve is that, often, a polynomial is an approximation to some other function, which is only valid in some closed interval (of the parameter). Outside this interval, the actual function and the polynomial are far apart from each other. The convex hull of the full curve possibly includes points that can only be obtained as convex combinations that include one of those faraway points, but given that the faraway points have nothing to do with the actual function, it is meaningless to have the convex hull to include them. 

The problem might also truncate the curve because certain values of a parameter might not make sense in the context of a problem. For instance, in \cite{Part1}, the problem motivating this, one cannot use more resources than one has available, which creates a hard cutoff.

\begin{thmb}
Every point in the convex hull of the curve that is the image of the parametric function $C_N: \left[t_{min}, t_{max}\right] \rightarrow \Re^N$, such that
$C_N\left(t\right) = 
\left(
\begin{matrix}
t \\
t^2 \\
t^3 \\
\cdots \\
t^N
\end{matrix}
\right)
$, can be represented as a convex combination of at most $\frac{N+1}{2}$ points on this curve if $N$ is odd, or as a convex combination of $\frac{N+2}{2}$ points on this curve if $N$ is even. Furthermore, if $N$ is even, one of these at most $\frac{N+2}{2}$ points on the curve can be required to be the point $C_N\left(t_{min}\right)$.
\end{thmb}
Section \ref{existencesection} provides a proof of the theorem. Section \ref{uniquenesssection} provides a proof of the claim that the representation of the form guaranteed to exist by the theorem is in fact unique, as long as the points are listed in ascending order of $t$, all terms with coefficient zero are removed, and all terms with the same point are merged. Section \ref{homeomorphismsection} provides a proof of the claim that, once this merging is done, then the evaluation of the convex combination indicated is a homeomorphism. The proofs in this paper are informal.

A motivation for this theorem is the following corollary, which is used in \cite{Part1}.
\begin{cor}
The same is true if $C_N\left(t\right)$ were 
$
\left(
\begin{matrix}
\sum \limits _{s = 0} ^{N} {a_{1s}t^s} \\
\sum \limits _{s = 0} ^{N} {a_{2s}t^s} \\
\sum \limits _{s = 0} ^{N} {a_{3s}t^s} \\
\cdots \\
\sum \limits _{s = 0} ^{N} {a_{Ns}t^s} \\
\end{matrix}
\right)
$, with all the $a_{is}$ real constants. That is, if all the entries of $C_N$ are polynomials in $t$ with degree at most $N$, then every point in the convex hull of $C_N$ still can be represented as a convex combination of $\frac{N+1}{2}$ points on this curve if $N$ is odd, or as a convex combination of $\frac{N+2}{2}$ points on this curve if $N$ is even. Furthermore, if $N$ is even, one of these $\frac{N+2}{2}$ can still be required to be the point $C_N\left(t_{min}\right)$.
\end{cor}
\begin{proof}
This curve is the result of a linear transformation on the original $C_N$. Any point in the convex hull of the new $C_N$ is thus a convex combination of points, each of which is the result of the same linear transformation applied to points on the original $C_N$. The linear transformation can be factored out, and the other factor becomes a convex combination of points on the original $C_N$, which can be reduced to a convex combination of $\frac{N+1}{2}$ points on the original $C_N$ if $N$ is odd, or as a convex combination of $\frac{N+2}{2}$ points on the original $C_N$ with one of points being at $t_{min}$ if $N$ is even. After the reduction, the linear transformation can be distributed again, leaving a representation as a convex combination of $\frac{N+1}{2}$ points on the new $C_N$ if $N$ is odd, or as a convex combination of $\frac{N+2}{2}$ points on the new $C_N$ with one of points being at $t_{min}$ if $N$ is even. This still evaluates to the same point as the convex combination at the beginning, providing a representation in the form required by this lemma.
\end{proof}

\section{Existence}
\label{existencesection}
\begin{lem} [Lemma on Pseudo-Vandermonde Matrices]
\label{vandermondelem}
Let $N$ be a nonnegative integer, and let $q$ be an integer with $\frac{N+1}{2} \leq q \leq N+1$. Then, define the pseudo-Vandermonde matrix $V_{N+1}\left(u_1, u_2, ..., u_q\right)$ to be 
\[
\left(
\begin{matrix}
1 & 1 & \cdots & 1 & 0 & 0 & \cdots & 0  \\
u_1 & u_2 & \cdots & u_q & 1 & 1 & \cdots & 1 \\
u_1^2 & u_2^2 & \cdots & u_q^2 & 2u_1 & 2u_2 & \cdots & 2u_{N+1-q} \\
\cdots & \cdots & \cdots & \cdots & \cdots & \cdots & \cdots & \cdots \\
u_1^N & u_2^N & \cdots & u_q^N & Nu_1^{N-1} & Nu_2^{N-1} & \cdots & Nu_{N+1-q}^{N-1} \\
\end{matrix}
\right)
\]
In other words, the first $q$ columns of $V_{N+1}\left(u_1, u_2, ..., u_q\right)$ are the first $q$ columns of a Vandermonde matrix of size $N+1$, while the last $N+1-q$ columns are the derivatives of the first $N+1-q$ columns. Then, the pseudo-Vandermonde matrix $V_{N+1}\left(u_1, u_2, ..., u_q\right)$ is nonsingular, as long as all the $u_j$ are pairwise distinct.
\end{lem}
\begin{proof}
Perform the following operations on the determinant of $V_{N+1}\left(u_1, u_2, \cdots, u_q\right)$:

1) Zero out each entry of the first column other than the first, by subtracting $u_1$ times each row from the row immediately below, in the bottom-to-top order.

2) Expand the determinant along the first column, leaving
\[
\hspace{-1 in}\det \left(
\begin{matrix}
u_2 - u_1 & \cdots & u_q - u_1 & 1 & \cdots & 1 \\
u_2^2 - u_1 u_2 & \cdots & u_q^2 - u_1 u_q & 2u_1 - u_1 & \cdots & 2u_{N+1-q} - u_1 \\
u_2^3 - u_1 u_2^2 & \cdots & u_q^3 - u_1 u_q^2 & 3u_1^2 - 2 u_1 u_1 & \cdots & 3u_{N+1-q} - 2 u_1 u_{N+1-q} \\
\cdots & \cdots & \cdots & \cdots & \cdots & \cdots \\
u_2^N - u_1 u_2^{N-1} & \cdots & u_q^N - u_1 u_q^{N-1} & Nu_1^{N-1} - \left(N-1\right) u_1 u_1^{N-2} & \cdots & Nu_{N+1-q}^{N-1} - \left(N-1\right) u_1 u_{N+1-q}^{N-2} \\
\end{matrix}
\right)
\]

3) Divide the first $q-1$ coumns by their top entry, which does not change whether the determinant is zero (as the $u_j$ are pairwise distinct), leaving
\[
\hspace{-0.8 in}\det \left(
\begin{matrix}
1 & \cdots & 1 & 1 & \cdots & 1 \\
u_2 & \cdots & u_q & 2u_1 - u_1 & \cdots & 2u_{N+1-q} - u_1 \\
u_2^2 & \cdots & u_q^2 & 3u_1^2 - 2 u_1 u_1 & \cdots & 3u_{N+1-q} - 2 u_1 u_{N+1-q} \\
\cdots & \cdots & \cdots & \cdots & \cdots & \cdots \\
u_2^{N-1} & \cdots & u_q^{N-1} & Nu_1^{N-1} - \left(N-1\right) u_1 u_1^{N-2} & \cdots & Nu_{N+1-q}^{N-1} - \left(N-1\right) u_1 u_{N+1-q}^{N-2} \\
\end{matrix}
\right)
\]
or, equivalently,
\[
\hspace{-0.8 in}\det \left(
\begin{matrix}
1 & \cdots & 1 & 1 & 1 & \cdots & 1 \\
u_2 & \cdots & u_q & u_1 & 2u_2 - u_1 & \cdots & 2u_{N+1-q} - u_1 \\
u_2^2 & \cdots & u_q^2 & u_1^2 & 3u_2^2 - 2 u_1 u_2 & \cdots & 3u_{N+1-q}^2 - 2 u_1 u_{N+1-q} \\
\cdots & \cdots & \cdots & \cdots & \cdots & \cdots & \cdots \\
u_2^{N-1} & \cdots & u_q^{N-1} & u_1^{N-1} & Nu_2^{N-1} - \left(N-1\right) u_1 u_2^{N-2} & \cdots & Nu_{N+1-q}^{N-1} - \left(N-1\right) u_1 u_{N+1-q}^{N-2} \\
\end{matrix}
\right)
\]

4) Subtract the $u_2$ through $u_q$ columns from the corresponding $2u_2-u_1$ through $2u_q - u_1$ columnns (if they exist), leaving
\[
\hspace{-1.1 in}\det \left(
\begin{matrix}
1 & \cdots & 1 & 1 & 0 & \cdots & 0 \\
u_2 & \cdots & u_q & u_1 & u_2 - u_1 & \cdots & u_{N+1-q} - u_1 \\
u_2^2 & \cdots & u_q^2 & u_1^2 & 2u_2^2 - 2 u_1 u_2 & \cdots & 2u_{N+1-q}^2 - 2 u_1 u_{N+1-q} \\
\cdots & \cdots & \cdots & \cdots & \cdots & \cdots & \cdots \\
u_2^{N-1} & \cdots & u_q^{N-1} & u_1^{N-1} & \left(N-1\right)u_2^{N-1} - \left(N-1\right) u_1 u_2^{N-2} & \cdots & \left(N-1\right)u_{N+1-q}^{N-1} - \left(N-1\right) u_1 u_{N+1-q}^{N-2} \\
\end{matrix}
\right)
\]

5) Divide the columns beginning with a $0$ by their second entries (which are nonzero), leaving

\[
\det \left(
\begin{matrix}
1 & \cdots & 1 & 1 & 0 & \cdots & 0 \\
u_2 & \cdots & u_q & u_1 & 1 & \cdots & 1 \\
u_2^2 & \cdots & u_q^2 & u_1^2 & 2u_2 & \cdots & 2u_{N+1-q} \\
\cdots & \cdots & \cdots & \cdots & \cdots & \cdots & \cdots \\
u_2^{N-1} & \cdots & u_q^{N-1} & u_1^{N-1} & \left(N-1\right)u_2^{N-2} & \cdots & \left(N-1\right)u_{N+1-q}^{N-2} \\
\end{matrix}
\right)
\]

but this is a smaller pseudo-Vandermonde matrix, still with all of the first $q$ columns distinct.

The $u_1$ column might not exist, but then, $q = N+1$, so steps 4 and 5 did nothing, and the result is
\[
\det \left(
\begin{matrix}
1 & \cdots & 1 \\
u_2 & \cdots & u_q \\
u_2^2 & \cdots & u_q^2 \\
\cdots & \cdots & \cdots \\
u_2^{N-1} & \cdots & u_q^{N-1} \\
\end{matrix}
\right)
\]
In either case, the result is the determinant of a smaller pseudo-Vandermonde matrix, which is zero if and only if the determinant of the original pseudo-Vanderminde matrix was zero. Eventually, this gets down to a $1$-by-$1$ matrix, where $N = 0$ and $q = 1$ (no other integer $q$ satisfies $\frac{0+1}{2} \leq q \leq 0 + 1$), and the single entry of this matrix is $1$. That matrix has determinant $1$, so the original pseudo-Vandermonde matrix has a nonzero determinant, and is thus nonsingular.

\end{proof}
\begin{defn}
For any positive integer $N$, and for any two real numbers $t_{min}$ and $t_{max}$, define \emph{the curve} $C_N$ to be image of the parametric function $C_N: \left[t_{min}, t_{max}\right] \rightarrow \Re^N$, such that
$C_N\left(t\right) = 
\left(
\begin{matrix}
t \\
t^2 \\
t^3 \\
\cdots \\
t^N
\end{matrix}
\right)
$. 
\end{defn}

\begin{defn}
Define the convex hull of $C_N$ to be the set of all finite convex combinations of points on $C_N$. A naming of a point
$
\left(
\begin{matrix}
v_1 \\
v_2 \\
v_3 \\
\cdots \\
v_N
\end{matrix}
\right)
$
is a representation of that point as a convex combination of points on $C_N$ of the form
$
\sum \limits _{j = 1} ^{M} {c_j
\left(
\begin{matrix}
t_j \\
t_j^2 \\
t_j^3 \\
\cdots \\
t_j^N
\end{matrix}
\right)
}
$
(which, when evaluated, results in the point
$
\left(
\begin{matrix}
v_1 \\
v_2 \\
v_3 \\
\cdots \\
v_N
\end{matrix}
\right)
$
). The terms are to be ordered in increasing order of $t_j$. The same naming can also be represented as a tuple, $\left(c_1, \cdots, c_N, t_1, \cdots, t_N\right)$.
\end{defn}
\begin{defn}
A \emph{naming} is a naming of some point in the convex hull of $C_N$.
\end{defn}
\begin{defn}
For any positive integer $M$, an $M$-naming is a naming that uses $M$ points.
\end{defn}
\begin{defn}
For any positive integer $M$, define an \emph{$\left(M - \frac{1}{2}\right)$-naming} as an $M$-naming with $t_1 = t_{min}$. (It is called that way, because there is a $\left(2M-2\right)$-parameter family of $\left(M - \frac{1}{2}\right)$-namings, while there is a $\left(2M-3\right)$-parameter family of $\left(M-1\right)$-namings and there is a $\left(2M-1\right)$-parameter family of $M$-namings.)
\end{defn}
\begin{defn}
For any positive integer $M$, let an $M$-naming be \emph{reducible} if and only if, either two adjacent terms use the same point, or a term has coefficient $0$. Let an $\left(M-\frac{1}{2}\right)$-naming be \emph{reducible} if and only if, either two adjacent terms use the same point, or a term has coefficient $0$.
(Note that $c_1 = 0$ does not suffice to make an $\left(M - \frac{1}{2}\right)$-naming reducible. Also note that, because the terms are in increasing order of $t$, then if two non-adjacent terms use the same point, then all terms between them use that point also, so the $M$-naming is still reducible.)
\end{defn}
\begin{lem} [Dimension Reduction Lemma]
\label{dimreductionlem}
For any positive integer $N$, and for any positive integer or positive half-integer $M'$, any point in the convex hull of $C_N$ that has a reducible $M'$-naming also has an $\left(M'-1\right)$-naming (no claim is made about whether or not this $\left(M'-1\right)$-naming is reducible).
\end{lem}
\begin{proof}
Let $M = M'$ if $M'$ is an integer or $M = M' + \frac{1}{2}$ if $M'$ is a half-integer. (Note that an $M'$-naming is an $M$-naming in either case.) If two adjacent terms use the same point, then they can be combined into a single term. If a term (not the first one if $M'$ is a half-integer) has coefficient zero, then it can be removed. The result is an $M$-naming, and if $M'$ is a half-integer, then its first point is still at $t_{min}$, because it was not removed, so the result is an $M'$-naming. (The points are still in increasing order of $t_j$, because any merger of two terms is only a merger of two points with the same $t$, which are adjacent to each other.)
\end{proof}

\begin{defn}
For any positive integer or a positive half-integer $M'$, let an $l$-boundary $M'$-naming be a naming that is on exactly $l$ boundaries of the naming space. Any zero coefficient counts as one boundary. Any two adjacent terms using the same point count as one boundary. If the first term uses the point at $t_{min}$, that counts as one boundary if $M'$ is an integer. (It does not count if $M'$ is a half-integer, because it is a required condition for every $M'$-naming in this case.) If the last term uses the point at $t_{max}$, that counts as one boundary.
\end{defn}

\begin{rmk}
The name \emph{$l$-boundary naming} comes from the fact that each of the statements is the equality case of one of the inequalities that are requirements to be an $M$-naming or $\left(M-\frac{1}{2}\right)$-naming. These are the boundaries of the space of the namings, so an $l$-boundary naming is a naming that is on $l$ boundaries of the naming space.
\end{rmk}

\begin{defn}
For any positive integer or positive half-integer $M$, let an \emph{interior} $M$-naming be a $0$-boundary $M$-naming, and let a \emph{boundary} $M$-naming be an $l$-boundary $M$-naming, with $l>0$.
\end{defn}

\begin{rmk}
From the definition of boundary $M$-namings, every reducible $M$-naming is a boundary $M$-naming, but there are boundary $M$-namings that are not reducible $M$-namings, namely those using the point at $t_{max}$, those using the point at $t_{min}$ if $M'$ is an integer, and those with first coefficient zero if $M'$ is a half-integer.
\end{rmk}

\begin{lem} [Lemma on Non-Reducible Boundary Namings]
\label{nonreducibleboundarylem}
For any positive integer $M$:

1) the non-reducible $1$-boundary $M$-namings satisfy either $t_1 = t_{min}$ or $t_M = t_{max}$, but not both;

2) the non-reducible $2$-boundary $M$-namings satisfy both $t_1 = t_{min}$ and $t_M = t_{max}$; and

3) there are no non-reducible $l$-boundary $\left(M - \frac{1}{2}\right)$-namings with $l \geq 3$;

4) the non-reducible $1$-boundary $\left(M - \frac{1}{2}\right)$-namings satisfy either $t_M = t_{max}$ or $c_1 = 0$, but not both;

5) the non-reducible $2$-boundary $\left(M - \frac{1}{2}\right)$-namings satisfy both $t_M = t_{max}$ and $c_1 = 0$; and

6) there are no non-reducible $l$-boundary $\left(M - \frac{1}{2}\right)$-namings with $l \geq 3$.
\end{lem}
\begin{proof}
The only boundaries that do not make an $M$-naming reducible are the first term using the point at $t_{min}$, and the last term using the point at $t_{max}$, so the non-reducible $1$-boundary $M$-namings are on exactly one of these boundaries and the non-reducible $2$-boundary $M$-namings are on both of them, and there can be no non-reducible $l$-boundary $M$-namings with $l \geq 3$. This shows 1), 2), and 3).

Similarly, the only boundaries that do not make an $\left(M-\frac{1}{2}\right)$-naming reducible are the first term having coefficient zero, and the last term using the point at $t_{max}$, so the non-reducible $1$-boundary $\left(M-\frac{1}{2}\right)$-namings are on exactly one of these boundaries and the non-reducible $2$-boundary $\left(M-\frac{1}{2}\right)$-namings are on both of them, and there can be no non-reducible $l$-boundary $\left(M-\frac{1}{2}\right)$-namings with $l \geq 3$. This shows 4), 5), and 6).
\end{proof}
\begin{lem} [Lemma on Neighborhoods of Namings]
\label{neighborhoodlem}
Let $N$ be a positive integer. Then:

1) around each interior $\left(\frac{N+3}{2}\right)$-naming of a specific point, there is a $2$-dimensional differentiable neighborhood of $\left(\frac{N+3}{2}\right)$-namings of that point, and

2) around each $1$-boundary $\left(\frac{N+3}{2}\right)$-naming of a specific point on a given boundary, there is a $1$-dimensional differentiable neighborhood of $\left(\frac{N+3}{2}\right)$-namings of that point that are on that boundary.
\end{lem}
\begin{rmk}
This is what justifies the use of Lagrange multipliers in Lemma \ref{nolocalextremalem}, the Lack of Local Extrema Lemma.
\end{rmk}
\begin{proof}
Let $P$ be a non-reducible $\left(\frac{N+3}{2}\right)$-naming of a point
$
\left(
\begin{matrix}
v_1 \\
v_2 \\
\cdots \\
v_n \\
\end{matrix}
\right)
$
. Also, let $M$ be $\frac{N+3}{2}$ or $\frac{N+4}{2}$, whichever of these is an integer. Then, $P$, which is an $\left(\frac{N+3}{2}\right)$-naming is an $M$-naming (with its $t_1$ equal to $0$ if $M  = \frac{N+4}{2}$). Let $P$ be the convex combination 
$
\sum \limits _{j = 1} ^{M} {c_j
\left(
\begin{matrix}
t_j \\
t_j^2 \\
\cdots \\
t_j^N
\end{matrix}
\right)
}
$.
Since this must evaluate to 
$
\left(
\begin{matrix}
v_1 \\
v_2 \\
\cdots \\
v_n \\
\end{matrix}
\right)
$,
and since the sum of the $c_j$ is $1$, it follows that
\begin{align*}
\sum \limits _{j = 1} ^{M} {c_j
\left(
\begin{matrix}
1 \\
t_j \\
t_j^2 \\
\cdots \\
t_j^N
\end{matrix}
\right)
}
=
\left(
\begin{matrix}
1 \\
v_1 \\
v_2 \\
\cdots \\
v_N
\end{matrix}
\right)
\end{align*}
Let $\gamma$ be the function from the set of all $\left(\frac{N+3}{2}\right)$-namings to the convex hull of $C_N$, with 
\[
\gamma\left(c_1, c_2, \cdots, c_M, t_1, t_2, \cdots, t_M\right) = \left( \sum \limits _{j = 1} ^{M} {c_j
\left(
\begin{matrix}
1 \\
t_j \\
t_j^2 \\
\cdots \\
t_j^N
\end{matrix}
\right)
}
\right)
-
\left(
\begin{matrix}
1 \\
v_1 \\
v_2 \\
\cdots \\
v_N
\end{matrix}
\right)
\]
Thus, the vector condition becomes simply $\gamma\left(c_1, c_2, \cdots, c_M, t_1, t_2, \cdots, t_M\right) = 0$ (with $0$ being the $N$-dimensional zero vector).
For all $i \in \left \lbrace 0, 1, \cdots, N \right \rbrace$, let
\[
\gamma_i\left(c_1, c_2, \cdots, c_M, t_1, t_2, \cdots, t_M\right) = \left( \sum \limits _{j = 1} ^{M} {c_j t_j^i}\right)- v_i
\]
which means that the $\gamma_i$ are the components of $\gamma$ as a vector. For convenience, $v_0$ was defined to be $1$.

The partial derivatives of $\gamma_i$ are:
\[
\frac{\partial \gamma_i}{\partial c_{j'}} = t_{j'}^i \\
\]
and
\[
\frac{\partial \gamma_i}{\partial t_{j'}} = c_{j'}\left(i t_{j'}^{i-1}\right) \\
\]
If $M = \frac{N+4}{2}$, or if $P$ is a boundary $\left(\frac{N+3}{2}\right)$-naming, or both, then some of these derivatives are not necessary, because the variable they are with respect to is held constant, as in the following table (which also defines the \emph{type} of each case):

\begin{tabular}{|c|c|c|c|c|}
\hline 
$M$ & Boundary count & Boundary equation & Variables held constant & Type \\ 
\hline 
$\frac{N+3}{2}$ & $0$ & none & none & 1 \\ 
$\frac{N+3}{2}$ & $1$ & $t_1 = t_{min}$ & $t_1$ & 2a \\ 
$\frac{N+3}{2}$ & $1$ & $t_M = t_{max}$ & $t_M$ & 2b \\ 
$\frac{N+4}{2}$ & $0$ & none & $t_1$ & 3 \\ 
$\frac{N+4}{2}$ & $1$ & $t_M = t_{max}$ & $t_1$ and $t_M$ & 4a \\ 
$\frac{N+4}{2}$ & $1$ & $c_1 = 0$ & $t_1$ and $c_1$ & 4b \\ 
\hline 
\end{tabular}

By Lemma \ref{nonreducibleboundarylem}, the Lemma on Non-Reducible Boundary Namings, the $\left(\frac{N+3}{2}\right)$-namings of types 2a, 2b, 4a, and 4b are the only non-reducible $1$-boundary $\left(\frac{N+3}{2}\right)$-namings.

In any case, the matrix of partial derivatives (with the rows indicating which of the $\gamma_i$ is differentiated, and the columns indicating the variable with which the derivative is taken, in the order $c_1, c_2, \cdots, c_M, t_1, t_2, \cdots, t_M$) is:
\[
\left(
\begin{matrix}
1 & 1 & \cdots & 1 & 0 & 0 & \cdots & 0 \\
t_1 & t_2 & \cdots & t_M & c_1 & c_2 & \cdots & c_M \\
t_1^2 & t_2^2 & \cdots & t_M^2 & 2 c_1 t_1 & 2 c_2 t_2 & \cdots & 2 c_M t_M \\
t_1^3 & t_2^3 & \cdots & t_M^3 & 3 c_1 t_1^2 & 3 c_2 t_2^2 & \cdots & 3 c_M t_M^2 \\
\cdots & \cdots & \cdots & \cdots & \cdots & \cdots & \cdots & \cdots \\
t_1^N & t_2^N & \cdots & t_M^N & N c_1 t_1^{N-1} & N c_2 t_2^{N-1} & \cdots & N c_M t_M^{N-1} \\
\end{matrix}
\right)
\]
with zero, one, or two, of its columns removed, depending on the type of $P$. It will now be shown that this matrix has full rank. 

For convenience, call the columns by the name of the variable with respect to which the partial derivative was taken to get that column. The names of the removed columns are exactly the variables that are held constant.

Regardless of the type of $P$, the rank of the matrix is unaffected by a multiplication (or division) of one of its columns by a non-zero constant.  So, without affecting the rank, each non-removed $t_j$ column can be divided by $c_j$, as no $c_j$ is zero, or else $P$ would have been a reducible $\left(\frac{N+3}{2}\right)$-naming (except when $P$ is of type 4b, but then, $t_1$ was held constant, so the $t_1$ column was already removed).

This leaves the matrix
\[
\left(
\begin{matrix}
1 & 1 & \cdots & 1 & 0 & 0 & \cdots & 0 \\
t_1 & t_2 & \cdots & t_M & 1 & 1 & \cdots & 1 \\
t_1^2 & t_2^2 & \cdots & t_M^2 & 2 t_1 & 2 t_2 & \cdots & 2 t_M \\
t_1^3 & t_2^3 & \cdots & t_M^3 & 3 t_1^2 & 3 t_2^2 & \cdots & 3 t_M^2 \\
\cdots & \cdots & \cdots & \cdots & \cdots & \cdots & \cdots & \cdots \\
t_1^N & t_2^N & \cdots & t_M^N & N t_1^{N-1} & N t_2^{N-1} & \cdots & N t_M^{N-1} \\
\end{matrix}
\right)
\]
with some columns removed. This matrix has $N+1$ rows and at least $N+1$ columns (it had $2M \geq N+3$ columns before the removals, with at most $2$ columns removed). Thus, it has at least as many columns as it has rows. If some (perhaps none) columns can be removed to make this a square matrix that is not singular, then the matrix before this second removal has full rank. Which additional columns are removed depends on the type of $P$.

\begin{tabular}{|c|c|c|c|c|}
\hline 
Type & Removed columns & $M$ & Columns remaining & Further removed columns \\ 
\hline 
1 & none & $\frac{N+3}{2}$ & $N+3$ & $c_1$ and $t_1$  \\ 
2a & $t_1$ & $\frac{N+3}{2}$ & $N+2$ & $c_1$  \\ 
2b & $t_M$ & $\frac{N+3}{2}$ & $N+2$ & $c_M$  \\ 
3 & $t_1$ & $\frac{N+4}{2}$ & $N+3$ & $c_1$ and $t_M$  \\ 
4a & $t_1$ and $t_M$ & $\frac{N+4}{2}$ & $N+2$ & $c_1$  \\ 
4b & $t_1$ and $c_1$ & $\frac{N+4}{2}$ & $N+2$ & $t_M$  \\ 
\hline 
\end{tabular}

For types 1, 2a, and 2b, one of the $c_j$ columns and the corresponding one of the $t_j$ columns was removed, so the remaining matrix is a pseudo-Vandermonde matrix (since it is a square matrix and each ``derivative'' column has a corresponding ``
$
\left(
\begin{matrix}
1 \\
t_j \\
\cdots \\
t_j^N \\
\end{matrix}
\right)
$
'' column). For types 3, 4a, and 4b, one such pair of columns, along with another $t_j$ column, is removed, so, again, the remaining matrix is a pseudo-Vandermonde matrix. The $t_j$ are all distinct, so the remaining matrix is nonsingular and thus of full rank (rank $N+1$). Adding back the further removed columns leaves the rank at $N+1$, so the matrix still has full rank.

That means, by the Implicit Function Theorem \cite{WeissteinImplicit}, that, when $\gamma\left(c_1, \cdots, c_M, t_1, \cdots, t_N\right)$ is held at zero and the variables indicated by the type of $P$ are held constant according to type, all variables ($c_1, c_2, \cdots, c_M, t_1, t_2, \cdots, c_M$ except $t_1$ if $M = \frac{N+4}{2}$) are differentiable functions of the further-removed variables, or of the one further-removed variable (call them $y_1$ and $y_2$, or just $y_1$, as the case may be), in some neighborhood of $P\left(y_1\right)$ and $P\left(y_2\right)$, or just of $P\left(y_1\right)$, as the case may be. Furthermore, the number of variables is exactly as this lemma requires: $2$ if $P$ is of type $1$ or $3$ (a $0$-boundary naming), and $1$ if $P$ is of any other type (a $1$-boundary naming). 

This provies a differentiable neighborhood of tuples around $P$ with $\gamma = 0$. 

For all types of $P$:

\begin{tabular}{|c|c|c|c|}
\hline 
Type & $M$ & Fixed equalities & Boundaries \\
\hline
1 & $\frac{N+3}{2}$ & none & $0$ \\
2a & $\frac{N+3}{2}$ & $t_1 = t_{min}$ & $1$ \\
2b & $\frac{N+3}{2}$ & $t_M = t_{max}$ & $1$ \\
3 & $\frac{N+4}{2}$ & $t_1 = t_{min}$ & $0$ \\
4a & $\frac{N+4}{2}$ & $t_1 = t_{min}$ and $t_M = t_{max}$ & $1$ \\
4b & $\frac{N+4}{2}$ & $t_1 = t_{min}$ and $c_1 = 0$ & $1$ \\
\hline 
\end{tabular}

so there is no room for any other of $t_1 = 0$, or $t_M = 0$, or $c_1 = 0$ to hold for $P$, as that would increase the boundary count of the $\left(\frac{N+3}{2}\right)$-naming $P$. $P$ cannot be on any other boundary (no other $c_j$ can be zero and no two adjacent $t_j$ can be equal). Thus, if the neighborhood is small enough, the tuples in the neighborhood are $M$-namings of 
$
\left(
\begin{matrix}
v_1 \\
v_2 \\
\cdots \\
v_N \\
\end{matrix}
\right)
$,
as $\gamma\left(c_1, c_2, \cdots, c_M, t_1, t_2, \cdots, t_M\right) = 0$ ensures that they evaluate to that point and have the coefficients sum to $1$, and the coefficients (except $c_1$ if $N$ is even) are positive (because the coefficients of $P$ are positive) and are in strictly-increasing order (because that is true for $P$), and the $t_j$ (other than the exceptions in the table) are in the open interval $\left(t_{min}, t_{max}\right)$ (because that is true for $P$), and because any $c_j$ or $t_j$ that is an exception is at the same value as the same $c_j$ or $t_j$ of $P$.

Furthermore, if $\frac{N+3}{2}$ is a half-integer (so $M = \frac{N+4}{2}$), then the fixed $t_1 = 0$ equality guarantees that the $M$-namings in the neighborhood of $P$ are $\left(\frac{N+3}{2}\right)$-namings. If $\frac{N+3}{2}$ is an integer (so $M = \frac{N+3}{2}$), then the $M$-namings in the neighborhood of $P$ are automatically $\left(\frac{N+3}{2}\right)$-namings.

Also, if $P$ is a $1$-boundary $\left(\frac{N+3}{2}\right)$-naming, then the $\left(\frac{N+3}{2}\right)$-namings in the neighborhood of $P$ are likewise $1$-boundary namings, because they are on the same boundary as the boundary on which $P$ is. If, instead, $P$ is a $0$-boundary $\left(\frac{N+3}{2}\right)$-naming, then the $\left(\frac{N+3}{2}\right)$-namings in the neighborhood of $P$ are likewise $0$-boundary namings, because $P$ is at least a certain distance away from any boundary.

Thus, these tuples are, in fact, $\left(\frac{N+3}{2}\right)$-namings that are on the same boundaries on which $P$ is. So, in fact, there exists a $2$-dimensional differentiable neighborhood of $\left(\frac{N+3}{2}\right)$-namings, if $P$ is a $0$-boundary $\left(\frac{N+3}{2}\right)$-naming, or a $1$-dimensional differentiable neighborhood of $\left(\frac{N+3}{2}\right)$-namings, if $P$ is a $1$-boundary $\left(\frac{N+3}{2}\right)$-naming, every naming in which is on the same boundaries on which $P$ is. This is what this lemma claimed.
\end{proof}

\begin{lem} [Compactness Lemma]
\label{compactnesslem}
For all positive integers and positive half-integers $M'$, for all positive integers $N$, and for all points 
$
\left(
\begin{matrix}
v_1 \\
\cdots \\
v_N \\
\end{matrix}
\right)
$ in the convex hull of $C_N$, the set of all $M'$-namings of
$
\left(
\begin{matrix}
v_1 \\
\cdots \\
v_N \\
\end{matrix}
\right)
$
is compact.
\end{lem}
\begin{proof}
Let $M$ = $M'$ or $M' + \frac{1}{2}$, whichever of these is an integer. It suffices to show that this set is closed and bounded. Let $\left(c_1, \cdots, c_M, t_1, \cdots, t_N\right)$ be a naming of the point
$
\left(
\begin{matrix}
v_1 \\
\cdots \\
v_N
\end{matrix}
\right)
$. The set of tuples meeting each individual condition of being an $M$-naming of
$
\left(
\begin{matrix}
v_1 \\
\cdots \\
v_N
\end{matrix}
\right)
$ 
(that each coefficient be nonnegative, that each $t_j$ be not lower than the previous one, that each $t_j$ be in the closed interval $\left[t_{min}, t_{max}\right]$, that the sum of the coefficient be $1$, and that 
$
\sum \limits _{j = 1} ^{M} {c_j
\left(
\begin{matrix}
t_j \\
\cdots \\
t_j^N
\end{matrix}
\right)
}
=
\left(
\begin{matrix}
v_1 \\
\cdots \\
v_N
\end{matrix}
\right)
$
) is closed, so the set of $M$-namings of 
$
\left(
\begin{matrix}
v_1 \\
\cdots \\
v_N
\end{matrix}
\right)
$
is an intersection of closed sets, and is thus closed. The set of tuples meeting the additional constraint $t_1 = t_{min}$ is also closed, which makes the set of all $M'$-namings of 
$
\left(
\begin{matrix}
v_1 \\
\cdots \\
v_N
\end{matrix}
\right)
$ 
closed regardless of whether $M'$ is an integer or a half-integer (as the intersection of two closed sets is closed). The set of $M$-namings of $
\left(
\begin{matrix}
v_1 \\
\cdots \\
v_N
\end{matrix}
\right)
$ 
is bounded, because each $t_j$ (point on the curve $C_N$) is in the closed interval $\left[t_{min}, t_{max}\right]$, and because each $c_j$ (coefficient) is in the closed interval $\left[0, 1\right]$. The set of $M'$-namings of 
$
\left(
\begin{matrix}
v_1 \\
\cdots \\
v_N
\end{matrix}
\right)
$ 
is a subset of the set of $M$-namings of
$
\left(
\begin{matrix}
v_1 \\
\cdots \\
v_N
\end{matrix}
\right)
$, 
so it is likewise bounded. Thus, the set of $M'$-namings of 
$
\left(
\begin{matrix}
v_1 \\
\cdots \\
v_N
\end{matrix}
\right)
$ 
is closed and bounded, and therefore, it is compact.
\end{proof}
\begin{lem}
\label{reducibility01lem}
If $N = 0$ or $N = 1$, then any point in the convex hull of $C_N$ has a reducible $\left(\frac{N+3}{2}\right)$-naming.
\end{lem}
\begin{proof}
If $N = 0$, then $\frac{N+3}{2} = \frac{3}{2}$. $\left(1, 0, t_{min}, t_{max}\right)$ is a $2$-naming of
$
\left(
\begin{matrix}
v_1 \\
v_2 \\
v_3 \\
\cdots \\
v_N
\end{matrix}
\right)
$, as all the $t$s are in increasing order, and satisfy $t_{min} \leq t_j \leq t_{max}$, and as all the $c$s are nonnegative and sum to 1, and as
$
\sum \limits _{j = 1} ^{M} {c_j
\left(
\begin{matrix}
t_j \\
t_j^2 \\
t_j^3 \\
\cdots \\
t_j^N
\end{matrix}
\right)
}
=
\left(
\begin{matrix}
v_1 \\
v_2 \\
v_3 \\
\cdots \\
v_N
\end{matrix}
\right)
$ is satisfied automatically because both sides are vectors with no components. $\left(1, 0, t_{min}, t_{max}\right)$ is a $\left(\frac{3}{2}\right)$-naming, as $t_1 = t_{min}$; it is also reducible, because $c_2 = 0$, so the lemma holds for $N = 0$.

If $N = 1$, then $\frac{N+3}{2} = 2$. $\left(1, 0, v_1, t_{max}\right)$ is a $2$-naming of
$
\left(
\begin{matrix}
v_1 \\
v_2 \\
v_3 \\
\cdots \\
v_N
\end{matrix}
\right)
$
, as all the $t$s are in increasing order, and satisfy $t_{min} \leq t_j \leq t_{max}$ (as both of these statements hold because $t_{min} \leq v_1 \leq t_{max}$, and this is required because $C_N$ is the set $\left(v_1\right)$ with $t_{min} \leq v_1 \leq t_{max}$, which is its own convex hull), and as the $c$s are all nonnegative and sum to $1$, and as
$
\sum \limits _{j = 1} ^{M} {c_j
\left(
\begin{matrix}
t_j \\
t_j^2 \\
t_j^3 \\
\cdots \\
t_j^N
\end{matrix}
\right)
}
=
\left(
\begin{matrix}
v_1 \\
v_2 \\
v_3 \\
\cdots \\
v_N
\end{matrix}
\right)
$
is satisfied (because $N = 1$ and $\sum \limits _{j = 1} ^{2} {c_j t_j} = c_1 t_1 + c_2 t_2 = 1(v_1) + 0(t_{max}) = v_1$). $\left(1, 0, v_1, t_{max}\right)$ is reducible because $c_2 = 0$, so the lemma holds for $N = 1$ also.
\end{proof}
\begin{lem} [Lack of Local Extrema Lemma]
\label{nolocalextremalem}
For any point in the convex hull of $C_N$, no $0$-boundary $\left(\frac{N+3}{2}\right)$-naming of that point or non-reducible $1$-boundary $\left(\frac{N+3}{2}\right)$-naming of that point is a local maximum or a local minimum of any of its $t_j$ variables that were not fixed. "Local" means "in the $2$-dimensional or $1$-dimensional neighborhood that is guaranteed to exist by Lemma \ref{neighborhoodlem}, the Lemma on Neighborhoods of Namings".
\end{lem}
\begin{proof}
Let $M$ be $\frac{N+3}{2}$ or $\frac{N+4}{2}$, whichever of these is an integer. Also let $\left(c_1, c_2, \cdots, c_M, t_1, t_2, \cdots, t_M\right)$ be a $0$-boundary $\left(\frac{N+3}{2}\right)$-naming of 
$
\left(
\begin{matrix}
v_1
v_2
\cdots
v_N
\end{matrix}
\right)
$, or a non-reducible $1$-boundary $\left(\frac{N+3}{2}\right)$-naming of that point. By Lemma \ref{neighborhoodlem}, the Lemma on Neighborhoods of Namings, there is a $2$-dimensional or $1$-dimensional neighborhood of 
$
\left(
\begin{matrix}
v_1
v_2
\cdots
v_N
\end{matrix}
\right)
$ around $\left(c_1, c_2, \cdots, c_M, t_1, t_2, \cdots, t_M\right)$. (The number of dimensions depends on the number of boundaries: $2$ dimensions for a non-reducible $1$-boundary $\left(\frac{N+3}{2}\right)$-naming, or $1$ dimension for a non-reducible $1$-boundary $\left(\frac{N+3}{2}\right)$-naming.) 

Suppose that $\left(c_1, c_2, \cdots, c_M, t_1, t_2, \cdots, t_M\right)$ is a local maximum or a local minimum of $t_s$, where $s$ is such that $t_s$ is not fixed. Lagrange multipliers can be used to check if that is indeed the case.

Lemma \ref{nonreducibleboundarylem}, the Lemma on Non-Reducible Boundary Namings, shows that, if $M = \frac{N+3}{2}$, then all non-reducible $1$-boundary $\left(\frac{N+3}{2}\right)$-namings satisfy either $t_1 = t_{min}$ or $t_M = t_{max}$, and if $M = \frac{N+4}{2}$, then all non-reducible $1$-boundary $\left(\frac{N+3}{2}\right)$-namings satisfy either $t_M = t_{max}$ or $c_1 = 0$.

The possibilities for the number and equations of boundaries, as well as for whether $M$ is  $\frac{N+3}{2}$ or $\frac{N+4}{2}$ will be classified by type (which means exactly the same as in Lemma \ref{neighborhoodlem} (the Lemma on Neighborhoods of Namings). The type determines the fixed variables:

\begin{tabular}{|c|c|c|c|}
\hline 
Type & $M$ & Boundaries & Fixed variables\\
\hline
1 & $\frac{N+3}{2}$ & none & none \\
2a & $\frac{N+3}{2}$ & $t_1 = t_{min}$ & $t_1 = t_{min}$ \\
2b & $\frac{N+3}{2}$ & $t_M = t_{min}$ & $t_M = t_{max}$ \\
3 & $\frac{N+4}{2}$ & none & $t_1 = t_{min}$ \\
4a & $\frac{N+4}{2}$ & $t_M = t_{max}$ & $t_1 = t_{min}$ and $t_M = t_{max}$ \\
4b & $\frac{N+4}{2}$ & $c_1 = 0$ & $t_1 = t_{min}$ and $c_1 = 0$ \\
\hline 
\end{tabular}

The constraints on $\left(c_1, c_2, \cdots, c_M, t_1, t_2, \cdots, t_N\right)$ come from the definition of $M$-naming (any fixed variables will be considered to be constants, and so, they do not contribute additional constraints), and they are:
\begin{align*}
& \sum \limits _{j = 1} ^{M} {c_j} = 1 \\
& \sum \limits _{j = 1} ^{M} {c_j
\left(
\begin{matrix}
t_j \\
t_j^2 \\
t_j^3 \\
\cdots \\
t_j^N
\end{matrix}
\right)
}
=
\left(
\begin{matrix}
v_1 \\
v_2 \\
v_3 \\
\cdots \\
v_N
\end{matrix}
\right) \\
\end{align*}
which can be rewritten as a system of scalar constraints
\[
\forall _{i \in \left \lbrace 0, 1, \cdots, N \right \rbrace} \sum \limits _{j = 1} ^M {c_j t_j^i} = v_i
\]
with $v_0 = 1$.

Regardless of the type, Lagrange multipliers specify that the gradient of the optimized function (let $t_s$ be optimized) must equal a linear combination of the gradients of the constraints. Lagrange multipliers can be used, because each $0$-boundary naming and each non-reducible $1$-boundary naming has a neighborhood of the correct dimension (as shown in Lemma \ref{neighborhoodlem}, the Lemma on Neighborhoods of Namings) so:

\begin{align*}
\nabla t_{j'} &= \sum \limits _{i = 0} ^{N} {\left(\lambda_i \nabla \sum \limits _{j = 1} ^M {c_j t_j^i}\right)} \\
\left(
\begin{matrix}
\frac{\partial t_s}{\partial c_1} \\
\frac{\partial t_s}{\partial c_2} \\
\cdots \\
\frac{\partial t_s}{\partial c_M} \\
\frac{\partial t_s}{\partial t_1} \\
\frac{\partial t_s}{\partial t_2} \\
\cdots \\
\frac{\partial t_s}{\partial t_M} \\
\end{matrix}
\right)
&= \sum \limits _{i = 0} ^{N} {\left(\lambda_i
\left(
\begin{matrix}
\frac{\partial}{\partial c_1} \left[\sum \limits _{j = 1} ^M {c_j t_j^i}\right] \\
\cdots \\
\frac{\partial}{\partial c_M} \left[\sum \limits _{j = 1} ^M {c_j t_j^i}\right] \\
\frac{\partial}{\partial t_1} \left[\sum \limits _{j = 1} ^M {c_j t_j^i}\right] \\
\cdots \\
\frac{\partial}{\partial t_M} \left[\sum \limits _{j = 1} ^M {c_j t_j^i}\right] \\
\end{matrix}
\right)
\right)} \\
\end{align*}
The components whose variables were fixed are understood to be skipped. This simplifies to:
\begin{align*}
\forall j'_{c_{j'} \text{ is not fixed}} \frac{\partial t_s}{\partial c_{j'}} &= \sum \limits _{i = 0} ^{N} {\left(\lambda_i c_{j'} t_{j'}^i\right)} \\
\forall j'_{t_{j'} \text{ is not fixed}} \frac{\partial t_s}{\partial t_{j'}} &= \sum \limits _{i = 0} ^{N} {\left(i \lambda_i c_{j'} t_{j'}^{i-1}\right)} \\
\\
\forall j'_{c_{j'} \text{ is not fixed}} \frac{\partial t_s}{\partial c_{j'}} &= \sum \limits _{i = 0} ^{N} {\lambda_i t_{j'}^i} \\
\forall j'_{t_{j'} \text{ is not fixed}} \frac{1}{c_{j'}}\frac{\partial t_s}{\partial t_{j'}} &= \sum \limits _{i = 0} ^{N} {i \lambda_i  t_{j'}^{i-1}} \\
\end{align*}
The last step is justified, because if $c_{j'} = 0$, then that makes $\left(c_1, c_2, \cdots, c_M, t_1, t_2, \cdots, t_M\right)$ a reducible naming (unless $j' = 1$ and $M = \frac{N+4}{2}$, but $t_1$ is fixed then, so that equation was not considered in that case).

Let $P\left(z\right) = \sum \limits _{i = 0} ^{N} {\lambda_i z^i}$ be a polynomial of degree $N$ or less in $z$. Then, its derivative is $\frac{dP}{dz} = \sum \limits _{i = 0} ^{N} {i \lambda_i z^{i-1}}$. Therefore, the conditions simplify to:
\begin{align*}
\forall j'_{c_{j'} \text{ is not fixed}} \frac{\partial t_s}{\partial c_{j'}} &= P\left(t_{j'}\right) \\
\forall j'_{t_{j'} \text{ is not fixed}} \frac{1}{c_{j'}}\frac{\partial t_s}{\partial t_{j'}} &= \left(\frac{dP}{dz}\right)_{z = t_{j'}} \\
\end{align*}
All (except $\frac{\partial t_s}{\partial t_s}$ and the partial derivatives with respect to fixed variables) of the $\frac{\partial t_s}{\partial c_{j'}}$ and $\frac{\partial t_s}{\partial t_{j'}}$ are zero. Therefore, with some exceptions,
\begin{align*}
P\left(t_{j'}\right) &= 0 \\
\left(\frac{dP}{dz}\right)_{z = t_{j'}} &= 0 \\
\end{align*}
The exceptions depend on type, and are:

\begin{tabular}{|c|c|c|c|}
\hline 
Type & $M$ & Fixed variables & Exceptions\\
\hline
1 & $\frac{N+3}{2}$ & none & $t_s$ \\
2a & $\frac{N+3}{2}$ & $t_1 = t_{min}$ & $t_1$ and $t_s$ \\
2b & $\frac{N+3}{2}$ & $t_M = t_{max}$ & $t_s$ and $t_M$ \\
3 & $\frac{N+4}{2}$ & $t_1 = t_{min}$ & $t_1$ and $t_s$ \\
4a & $\frac{N+4}{2}$ & $t_1 = t_{min}$ and $t_M = t_{max}$ & $t_1$ and $t_s$ and $t_M$ \\
4b & $\frac{N+4}{2}$ & $t_1 = t_{min}$ and $c_1 = 0$ & $t_1$ and $t_s$ and $c_1$ \\
\hline 
\end{tabular}

Note that $t_s$ is not a duplicate of any other exception, because $s$ was defined to have $t_s$ not fixed.

If $j'$ is such that neither $c_{j'}$ nor $t_{j'}$ is an exception, then $P\left(t_{j'}\right)$ and $\left(\frac{dP}{dz}\right)_{z = t_{j'}} = 0$, so $t_{j'}$ is a double root of $P$. If $j'$ is such that only $t_{j'}$ is an exception, then $P\left(t_{j'}\right)$ and $\left(\frac{dP}{dz}\right)_{z = t_{j'}} = 0$, so $t_{j'}$ is a root of $P$. As all the $t$s are distinct (otherwise, $\left(c_1, c_2, \cdots, c_M, t_1, t_2, \cdots t_M\right)$ is reducible), these facts can be used to provide a lower bound on the number of roots of $P$. This table shows the minimum order of each root $t_{j'}$ (which could be $0$).

\begin{tabular}{|c|c|c|c|c|c|c|}
\hline 
Type & $M$ & Exceptions & Non-roots & Single roots & Total roots & Total roots\\
\hline
1 & $\frac{N+3}{2}$ & $t_s$ & none & $t_s$ & $2M-1$ & $N+2$ \\
2a & $\frac{N+3}{2}$ & $t_1$ and $t_s$ & none & $t_1$ and $t_s$ & $2M-2$ & $N+1$ \\
2b & $\frac{N+3}{2}$ & $t_s$ and $t_M$ & none & $t_s$ and $t_M$ & $2M-2$ & $N+1$ \\
3 & $\frac{N+4}{2}$ & $t_1$ and $t_s$ & none & $t_1$ and $t_s$ & $2M-2$ & $N+2$ \\
4a & $\frac{N+4}{2}$ & $t_1$ and $t_s$ and $t_M$ & none & $t_1$ and $t_s$ and $t_M$ & $2M-3$ & $N+1$ \\
4b & $\frac{N+4}{2}$ & $t_1$ and $t_s$ and $c_1$ & $t_1$ & $t_s$ & $2M-3$ & $N+1$ \\
\hline 
\end{tabular}

In each type, $P$ is an $N$-th degree polynomial with more than $N$ roots. Thus, $P$ is the zero polynomial. However, $\forall j'_{t_{j'} \text{ is not fixed}} \frac{1}{c_{j'}}\frac{\partial t_2}{\partial t_{j'}} = \left(\frac{dP}{dz}\right)_{z = t_{j'}}$ implies that $\frac{1}{c_2} = \left(\frac{dP}{dz}\right)_{z = t_2}$, but this statement cannot be correct, given that $P$ is the zero polynomial. 

Thus, the system of equations from the Lagrange multipliers leads to a contradiction, and thus, has no solutions. This means that $\left(c_1, c_2, \cdots, c_M, t_1, t_2, \cdots t_M\right)$ is neither a local maximum of $t_s$ nor a local minimum of $t_s$. Therefore, for any point in the convex hull of $C_N$, no $0$-boundary $\left(\frac{N+3}{2}\right)$-naming of that point or non-reducible $1$-boundary $\left(\frac{N+3}{2}\right)$-naming of that point is a local maximum or a local minimum of any of its $t_j$ variables that were not fixed.
\end{proof}
\begin{lem} [Uniqueness of the $2$-Boundary Non-reducible Naming Lemma]
\label{unique2boundarylem}
For a point in the convex hull of $C_N$, there is at most one non-reducible $2$-boundary $\left(\frac{N+3}{2}\right)$-naming.
\end{lem}
\begin{proof}
Let $M$ be $\frac{N+3}{2}$ or $\frac{N+4}{2}$, whichever of these is an integer. Let 
$
\sum \limits _{j = 1} ^{M} {c_j
\left(
\begin{matrix}
t_j \\
t_j^2 \\
\cdots \\
t_j^N
\end{matrix}
\right)
}
$
and
$
\sum \limits _{j = 1} ^{M} {c_j
\left(
\begin{matrix}
t_j \\
t_j^2 \\
\cdots \\
t_j^N
\end{matrix}
\right)
}
$
be two non-reducible $2$-boundary $\left(\frac{N+3}{2}\right)$-namings of the same point in the convex hull of $C_N$. Thus,
\[
\sum \limits _{j = 1} ^{M} {c_j
\left(
\begin{matrix}
1 \\
t_j \\
t_j^2 \\
\cdots \\
t_j^N
\end{matrix}
\right)
}
=
\sum \limits _{j = 1} ^{M} {c'_j
\left(
\begin{matrix}
1 \\
{t'}_j \\
{t'}_j^2 \\
\cdots \\
{t'}_j^N
\end{matrix}
\right)
}
\]
as the equality in the first components comes from the coefficients summing to $1$, and the equalities in the other components come from the namings evaluating to the same point.
By the definition of the half-integer naming, if $M = \frac{N+4}{2}$, then $t_1 = t_{min}$ and ${t'}_1 = t_{min}$. Further, by Lemma \ref{nonreducibleboundarylem}, the Lemma on Non-Reducible Boundary Namings, if $M = \frac{N+3}{2}$, then $t_1 = t_{min}$, ${t'}_1 = t_{min}$, $t_M = t_{max}$, and ${t'}_M = t_{max}$, and if $M = \frac{N+4}{2}$, then $t_M = t_{max}$, ${t'}_M = t_{max}$, $c_1 = 0$, and ${c'}_1 = 0$. In either case, $t_1 = t_{min}$, ${t'}_1 = t_{min}$, $t_M = t_{max}$, and ${t'}_M = t_{max}$ hold. Thus, the previous statement implies that:
\begin{align*}
& c_1
\left(
\begin{matrix}
1 \\
t_1 \\
t_1^2 \\
\cdots \\
t_1^N
\end{matrix}
\right)
+
\sum \limits _{j = 2} ^{M-1} {c_j
\left(
\begin{matrix}
1 \\
t_j \\
t_j^2 \\
\cdots \\
t_j^N
\end{matrix}
\right)
}
+
c_M
\left(
\begin{matrix}
1 \\
t_M \\
t_M^2 \\
\cdots \\
t_M^N
\end{matrix}
\right)
 =
c'_1
\left(
\begin{matrix}
1 \\
{t'}_1 \\
{t'}_1^2 \\
\cdots \\
{t'}_1^N
\end{matrix}
\right)
+
\sum \limits _{j = 2} ^{M-1} {c'_j
\left(
\begin{matrix}
1 \\
{t'}_j \\
{t'}_j^2 \\
\cdots \\
{t'}_j^N
\end{matrix}
\right)
}
+
c'_M
\left(
\begin{matrix}
1 \\
{t'}_M \\
{t'}_M^2 \\
\cdots \\
{t'}_M^N
\end{matrix}
\right)
\\
& c_1
\left(
\begin{matrix}
1 \\
t_{min} \\
t_{min}^2 \\
\cdots \\
t_{min}^N
\end{matrix}
\right)
+
\sum \limits _{j = 2} ^{M-1} {c_j
\left(
\begin{matrix}
1 \\
t_j \\
t_j^2 \\
\cdots \\
t_j^N
\end{matrix}
\right)
}
+
c_M
\left(
\begin{matrix}
1 \\
t_{max} \\
t_{max}^2 \\
\cdots \\
t_{max}^N
\end{matrix}
\right)
 =
c'_1
\left(
\begin{matrix}
1 \\
t_{min} \\
t_{min}^2 \\
\cdots \\
t_{min}^N
\end{matrix}
\right)
+
\sum \limits _{j = 2} ^{M-1} {c'_j
\left(
\begin{matrix}
1 \\
{t'}_j \\
{t'}_j^2 \\
\cdots \\
{t'}_j^N
\end{matrix}
\right)
}
+
c'_M
\left(
\begin{matrix}
1 \\
t_{max} \\
t_{max}^2 \\
\cdots \\
t_{max}^N
\end{matrix}
\right)
\\
& \left(c_1 - c'_1\right)
\left(
\begin{matrix}
1 \\
t_{min} \\
t_{min}^2 \\
\cdots \\
t_{min}^N
\end{matrix}
\right)
+
\sum \limits _{j = 2} ^{M-1} {c_j
\left(
\begin{matrix}
1 \\
t_j \\
t_j^2 \\
\cdots \\
t_j^N
\end{matrix}
\right)
}
-
\sum \limits _{j = 2} ^{M-1} {c'_j
\left(
\begin{matrix}
1 \\
{t'}_j \\
{t'}_j^2 \\
\cdots \\
{t'}_j^N
\end{matrix}
\right)
}
+
\left(c_M - c'_M\right)
\left(
\begin{matrix}
1 \\
t_{max} \\
t_{max}^2 \\
\cdots \\
t_{max}^N
\end{matrix}
\right)
=
\left(
\begin{matrix}
0 \\
0 \\
0 \\
\cdots \\
0 \\
\end{matrix}
\right)
\\
\end{align*}
Since, if $M = \frac{N+4}{2}$, then $c_1 = 0$ and $c'_1 = 0$ (as shown earlier), in that case, the first term disappears. This can be rewritten in matrix form:
\[ 
\left(
\begin{matrix}
1 & 1 & \cdots & 1 & 1 & \cdots & 1 & 1 \\
t_{min} & t_2 & \cdots & t_{M-1} & {t'}_2 & \cdots & {t'}_{M-1} & t_{max} \\
t_{min}^2 & t_2^2 & \cdots & t_{M-1}^2 & {t'}_2^2 & \cdots & {t'}_{M-1}^2 & t_{max}^2 \\
\cdots & \cdots & \cdots & \cdots & \cdots & \cdots & \cdots & \cdots \\
t_{min}^N & t_2^N & \cdots & t_{M-1}^N & {t'}_2^N & \cdots & {t'}_{M-1}^N & t_{max}^N \\
\end{matrix}
\right)
\left(
\begin{matrix}
c_1 - c'_1 \\
c_2 \\
\cdots \\
c_{M-1} \\
-c'_2 \\
\cdots \\
-c'_{M-1} \\
c_M - c'_M \\
\end{matrix}
\right)
=
\left(
\begin{matrix}
0 \\
0 \\
0 \\
\cdots \\
0 \\
\end{matrix}
\right)
\]
without its first column if $M = \frac{N+4}{2}$. The number of rows ($N+1$) and the number of columns ($2M-2$ if $M = \frac{N+3}{2}$, or $2M-2$ if $M = \frac{N+3}{2}$) are equal, so this is a square matrix. This is also a pseudo-Vandermonde matrix (in fact, it is a Vandermonde matrix). Thus, it is nonsingular, unless two of the $t$s and $t'$s are equal. 

If two $t$s or two $t'$s are equal (with $t_{min}$ and $t_{max}$ in both groups), one of the namings is reducible. If one $t$ and one $t'$ are equal, then in the vector form of the equation, merge the two terms, discard the last component of all the vectors, and reconvert to matrix form. The result is still a pseudo-Vandermonde matrix, and the $c$ vector has a $c$ and a $-c'$ component (with the same indices as the $t$ and the $t'$) replaced by the corresponding $c - c'$. 

As long as some $t$ equals some $t'$ (neither of them can be one of the $t$s in the already-merged terms, as that would be both a $t$ and a $t'$, and if it equals another $t$ or another $t'$, then one of the namings is reducible), repeat the procedure. Eventually, the $t$s would all be distinct. 

At this point, the $c$ vector is all zero, as the pseudo-Vandermonde matrix is nonsingular. If there were fewer than $M-2$ mergings, then in the $c$ vector, at least one of the $c$s is unmerged with a $c'$, and is thus zero, making 
$
\sum \limits _{j = 1} ^{M} {c_j
\left(
\begin{matrix}
t_j \\
t_j^2 \\
\cdots \\
t_j^N
\end{matrix}
\right)
}
$ 
reducible. (If $M = \frac{N+4}{2}$, then $c_1$ is not one of the $c$s in the $c$ vector, so it is another $c$ that is $0$.) If there were $M - 2$ mergings (and there cannot be more, because each merging involves one of the $t$s, and they are different from each other and from $t_{min}$ and $t_{max}$), then all the entries in the $c$ vector are $c - c'$ differences, each of which having coefficients the same as those of a merged pair of a $t$ and a $t'$. These differences are all zero, but that makes for a $c$ and a $c'$ equalling each other when the $t$ and the $t'$ with the same indices equal each other. These equalities cover all of $c_2, \cdots, c_{M-1}$ and $c_2, \cdots, c_{M-1}$, so 
$
\sum \limits _{j = 2} ^{M-1} {c_j
\left(
\begin{matrix}
1 \\
t_j \\
t_j^2 \\
\cdots \\
t_j^N
\end{matrix}
\right)
}
$ 
and 
$
\sum \limits _{j = 2} ^{M-1} {c'_j
\left(
\begin{matrix}
1 \\
{t'}_j \\
{t'}_j^2 \\
\cdots \\
{t'}_j^N
\end{matrix}
\right)
}
$ 
are the same sum (and the terms are not rearranged because the $t_j$ are in ascending order, as are the $t'_j$). Since $t_1 = t_{min} = t'_1$ and $t_M = t_{max} = t'_M$, and since $c_1 = c'_1$ (from $c_1 - c'_1 = 0$ if $M = \frac{N+3}{2}$, or from $c_1 = 0 = c'_1$ if $M = \frac{N+4}{2}$) and $c_M = c'_M$ (from $c_M - c'_M = 0$), it follows that the two namings 
$
\sum \limits _{j = 1} ^{M} {c_j
\left(
\begin{matrix}
1 \\
t_j \\
t_j^2 \\
\cdots \\
t_j^N
\end{matrix}
\right)
}
$ 
and 
$
\sum \limits _{j = 1} ^{M} {c'_j
\left(
\begin{matrix}
1 \\
{t'}_j \\
{t'}_j^2 \\
\cdots \\
{t'}_j^N
\end{matrix}
\right)
}
$ 
are term-for-term the same. Thus, no point in the convex hull of $C_N$ can have two distinct $2$-boundary $\left(\frac{N+3}{2}\right)$-namings.
\end{proof}
\begin{lem} [Non-Isolation of the $2$-Boundary Non-reducible Naming Lemma]
\label{nonisolationlemma}
If a point in the convex hull of $C_N$ has a non-reducible $2$-boundary $\left(\frac{N+3}{2}\right)$-naming, then that point has at least one other $\left(\frac{N+3}{2}\right)$-naming.
\end{lem}
\begin{proof}
Let $M$ be $\frac{N+3}{2}$ or $\frac{N+4}{2}$, whichever of these is an integer. For any positive real number $\epsilon$, let an \emph{$\epsilon$-extended $M$-naming} of a point 
$
\left(
\begin{matrix}
v_1 \\
v_2 \\
v_3 \\
\cdots \\
v_N \\
\end{matrix}
\right)
$ be the same thing as a naming of that point, except that the $t$-values are allowed to be in the closed interval $\left[t_{min}, t_{max} + \epsilon\right]$, rather than in $\left[t_{min}, t_{max}\right]$. Let $P = \left(c_1, c_2, \cdots, c_M, t_1, t_2, \cdots t_M\right)$ be a $2$-boundary $\left(\frac{N+3}{2}\right)$-naming of this point. Then, $P$ is an $\epsilon$-extended $\left(\frac{N+3}{2}\right)$-naming of the same point, as the only change in the conditions was an extension in the interval that the $t$s can be in. As an $\epsilon$-extended $\left(\frac{N+3}{2}\right)$-naming, $P$ is a $1$-boundary naming (with the boundary being $t_1 = t_{min}$ if $M = \frac{N+3}{2}$, or $c_1 = 0$ if $M = \frac{N+4}{2}$), as the extension moved the $t_M = t_{max}$ boundary away from this naming.

Thus, by Lemma \ref{neighborhoodlem}, the Lemma on Neighborhoods of Namings, $\left(c_1, c_2, \cdots, c_M, t_1, t_2, \cdots t_M\right)$ has a $1$-dimenstional neighborhood of $\epsilon$-extended $\left(\frac{N+3}{2}\right)$-namings of
$
\left(
\begin{matrix}
v_1 \\
v_2 \\
v_3 \\
\cdots \\
v_N
\end{matrix}
\right)
$ around it, that satisfy $t_1 = t_{min}$ if $M = \frac{N+3}{2}$, or that satisfy $c_1 = 0$ if $M = \frac{N+4}{2}$. It remains to be shown that at least some of these $\epsilon$-extended $\left(\frac{N+3}{2}\right)$-namings are in fact $\left(\frac{N+3}{2}\right)$-namings.

The only way that an $\epsilon$-extended $\left(\frac{N+3}{2}\right)$-naming would not be an $\left(\frac{N+3}{2}\right)$-naming is that $t_M > t_{max}$ (but still, $t_M \leq t_{max} + \epsilon$). The only way that none of the $\epsilon$-extended $\left(\frac{N+3}{2}\right)$-namings in this $1$-dimensional neighborhood (except $\left(c_1, c_2, \cdots, c_M, t_1, t_2, \cdots t_N\right)$ itself) would be $\left(\frac{N+3}{2}\right)$-namings is that $t_M \geq t_{max}$ in the whole neighborhood, which means that $t_M$ has a local minimum at $\left(c_1, c_2, \cdots, c_M, t_1, t_2, \cdots t_N\right)$. Lemma \ref{nolocalextremalem}, the Lack of Local Extrema Lemma, shows that $t_M$ (which is not fixed) is not a local minimum (or a local maximum) at $\left(c_1, c_2, \cdots, c_M, t_1, t_2, \cdots t_N\right)$. Thus, there is a contradiction, so, at least some of these $\epsilon$-extended $\left(\frac{N+3}{2}\right)$-namings are in fact $\left(\frac{N+3}{2}\right)$-namings.
\end{proof}

\begin{lem} [Existence of Reducible Namings Lemma]
\label{reduciblelem}
If a point in the convex hull of $C_N$ has an $\left(\frac{N+3}{2}\right)$-naming, then that point also has a reducible $\left(\frac{N+3}{2}\right)$-naming.
\end{lem}
\begin{proof}
If $N \leq 1$, then this was shown in Lemma \ref{reducibility01lem}, so only the $N \geq 2$ case needs to be considered here.  $N \geq 2$ means that either $M \geq \frac{5}{2}$ or $M \geq 3$, but as $M$ is an integer, $M \geq 3$. 
Let $M$ be $\frac{N+3}{2}$ or $\frac{N+4}{2}$, whichever of these is an integer.

The space of $\frac{N+3}{2}$-namings of a point 
$
\left(
\begin{matrix}
v_1 \\
v_2 \\
\cdots \\
v_N \\
\end{matrix}
\right)
$
is compact, by Lemma \ref{compactnesslem}, the Compactness Lemma (with $M' = \frac{N+3}{2}$).  Thus, any continuous function, including $t_2$, from this space to the real numbers must have a global minimum and a global maximum. Let $MAX$ be the global maximum and $MIN$ be the global minimum. $MAX$ (and $MIN$) is a $0$-boundary $\frac{N+3}{2}$-naming, or a non-reducible $1$-boundary $\frac{N+3}{2}$-naming, or a non-reducible $2$-boundary $\frac{N+3}{2}$-naming, or a reducible $\frac{N+3}{2}$-naming.

If $MAX$ or $MIN$ is a $0$-boundary $\left(\frac{N+3}{2}\right)$-naming or a $1$-boundary non-reducible $\left(\frac{N+3}{2}\right)$-naming, then, by Lemma \ref{neighborhoodlem}, the Lemma on Neighborhoods of Namings, there is a differentiable neighborhood of $\frac{N+3}{2}$-namings around $MAX$ or $MIN$. Thus, $MAX$ is a local maximum and $MIN$ is a local minimum. However, Lemma \ref{nolocalextremalem}, the Lack of Local Extrema Lemma, shows this not to be the case, so, neither $MAX$ nor $MIN$ is a $0$-boundary $\left(\frac{N+3}{2}\right)$-naming or a $1$-boundary non-reducible $\left(\frac{N+3}{2}\right)$-naming

If there is no reducible $\left(\frac{N+3}{2}\right)$-naming of $
\left(
\begin{matrix}
v_1 \\
v_2 \\
\cdots \\
v_N \\
\end{matrix}
\right)
$, then that means that both $MAX$ and $MIN$ are $2$-boundary non-reducible $\left(\frac{N+3}{2}\right)$-namings of this point. By Lemma \ref{unique2boundarylem}, the Uniqueness of the $2$-Boundary Non-reducible Naming Lemma, $MAX$ and $MIN$ are the same. That makes $t_2$ a constant. 

By Lemma \ref{nonisolationlemma}, the Non-Isolation of the $2$-Boundary Non-reducible Naming Lemma, there is another $\left(\frac{N+3}{2}\right)$-naming of this point than $MAX$ (or $MIN$) (call it $OTHER$). If $OTHER$ is reducible, then this lemma is correct. Otherwise, $OTHER$ is a $0$-boundary $\left(\frac{N+3}{2}\right)$-naming or a non-reducible $1$-boundary $\left(\frac{N+3}{2}\right)$-naming. ($OTHER$ cannot be a $2$-boundary $\left(\frac{N+3}{2}\right)$-naming of this point, because $MAX$, which is also $MIN$, is the only such $\left(\frac{N+3}{2}\right)$-naming.)

By Lemma \ref{neighborhoodlem}, the Lemma on Neighborhoods of Namings, there is a differentiable neighborhood of $\frac{N+3}{2}$-namings of $
\left(
\begin{matrix}
v_1 \\
v_2 \\
\cdots \\
v_N \\
\end{matrix}
\right)
$ around $OTHER$. Since $t_2$ is a constant, $OTHER$ is both a local maximum and a local minimum of $t_2$. However, Lemma \ref{nolocalextremalem}, the Lack of Local Extrema Lemma, shows this not to be the case.

Thus, the non-existence of reducible $\left(\frac{N+3}{2}\right)$-namings of this point leads to a contradiction. Thus, if a point in the convex hull of $C_N$ has an $\left(\frac{N+3}{2}\right)$-naming, then that point also has a reducible $\left(\frac{N+3}{2}\right)$-naming.
\end{proof}
\begin{lem} [Full Dimension Reduction Lemma]
\label{fulldimreductionlem}
If a point in the convex hull of $C_N$ has an $\left(\frac{N+3}{2}\right)$-naming, then that point has an $\left(\frac{N+1}{2}\right)$-naming.
\end{lem}
\begin{proof}
By Lemma \ref{reduciblelem}, the Existence of Reducible Namings Lemma, if a point in the convex hull of $C_N$ has an $\left(\frac{N+3}{2}\right)$-naming, then that point has a reducible $\left(\frac{N+1}{2}\right)$-naming. Therefore, by Lemma \ref{dimreductionlem}, the Dimension Reduction Lemma, that point has an $\left(\frac{N+1}{2}\right)$-naming.
\end{proof}
\begin{lem} [Generalized Dimension Reduction Lemma]
\label{gendimreductionlem}
For any positive integer or positive half-integer $M'$, such that $M' - \frac{N+3}{2}$ is a nonnegative integer, if a point in the convex hull of $C_N$ has an $M'$-naming, then that point also has an $\left(M'-1\right)$-naming.
\end{lem}
\begin{proof}
Let $M$ be $M'$ or $M' + \frac{1}{2}$, whichever of these is an integer, and let
$
\sum \limits _{j = 1} ^{M} {c_j
\left(
\begin{matrix}
t_1 \\
\cdots \\
t_N \\
\end{matrix}
\right)
}
$ 
be an $M'$-naming of a point
$
\left(
\begin{matrix}
v_1 \\
\cdots \\
v_N \\
\end{matrix}
\right)
$
in the convex hull of $C_N$. Let $Q$ be $\frac{N+3}{2}$ or $\frac{N+4}{2}$, whichever of these is an integer. The first $Q$ terms of this sum, rescaled to have their coefficients sum to $1$, form the sum
$
\sum \limits _{j = 1} ^Q {
\left( \frac{c_j}{\sum \limits _{s = 1} ^{Q} {c_s}}
\left(
\begin{matrix}
t_j \\
\cdots \\
t_j^N
\end{matrix}
\right)
\right)}
$, 
which is a $Q$-naming of some point in the convex hull of $C_N$ (as it is a convex combination of $Q$ points on $C_N$). This is also a $\left(\frac{N+3}{2}\right)$-naming (as, if $Q = \frac{N+4}{2}$, then $N$ is even, and thus, $M'$ is a half-integer, which means that $t_1 = t_{min}$, making this $\left(\frac{N+4}{2}\right)$-naming an $\left(\frac{N+3}{2}\right)$-naming). By Lemma \ref{fulldimreductionlem}, the Full Dimension Reduction Lemma, the point that is the result of evaluating 
$
\sum \limits _{j = 1} ^Q {
\left( \frac{c_j}{\sum \limits _{s = 1} ^{Q} {c_s}}
\left(
\begin{matrix}
t_j \\
\cdots \\
t_j^N
\end{matrix}
\right)
\right)}
$ 
has an $\left(\frac{N+1}{2}\right)$-naming, which can be written as  
$
\sum \limits _{j = 1} ^{Q-1} {d_j
\left(
\begin{matrix}
u_j \\
\cdots \\
u_j^N
\end{matrix}
\right)}
$. Then (with equality here meaning equality as algebraic expressions):
\begin{align*}
\sum \limits _{j = 1} ^M {c_j
\left(
\begin{matrix}
t_j \\
\cdots \\
t_j^N
\end{matrix}
\right)}
&=
\left(\sum \limits _{s = 1} ^{Q} {c_s}\right)
\sum \limits _{j = 1} ^Q {
\left( \frac{c_j}{\sum \limits _{s = 1} ^{Q} {c_s}}
\left(
\begin{matrix}
t_j \\
\cdots \\
t_j^N
\end{matrix}
\right)
\right)}
+
\sum \limits _{j = Q+1} ^M {c_j
\left(
\begin{matrix}
t_j \\
\cdots \\
t_j^N
\end{matrix}
\right)} \\
&=
\left(\sum \limits _{s = 1} ^{Q} {c_s}\right)
\sum \limits _{j = 1} ^{Q-1} {d_j
\left(
\begin{matrix}
u_j \\
\cdots \\
u_j^N
\end{matrix}
\right)}
+
\sum \limits _{j = Q+1} ^M {c_j
\left(
\begin{matrix}
t_j \\
\cdots \\
t_j^N
\end{matrix}
\right)} \\
&=
\sum \limits _{j = 1} ^{Q-1} {
\left(
\left(d_j\sum \limits _{s = 1} ^{Q} {c_s}\right)
\left(
\begin{matrix}
u_j \\
\cdots \\
u_j^N
\end{matrix}
\right)
\right)}
+
\sum \limits _{j = Q+1} ^M {c_j
\left(
\begin{matrix}
t_j \\
\cdots \\
t_j^N
\end{matrix}
\right)} \\
\end{align*}
This, after reordering its terms (in ascending order of the $t$ of the point used in each term), however, is a $\left(M-1\right)$-naming of 
$
\left(
\begin{matrix}
v_1 \\
\cdots \\
v_N \\
\end{matrix}
\right)
$, 
as it is a convex combination of $M-1$ points on $C_N$, and it evaluates to the point to which 
$
\sum \limits _{j = 1} ^M {c_j
\left(
\begin{matrix}
t_j \\
\cdots \\
t_j^N
\end{matrix}
\right)}
$ 
evaluates, which is 
$
\left(
\begin{matrix}
v_1 \\
\cdots \\
v_N \\
\end{matrix}
\right)
$. 
This is also a $\left(M'-1\right)$-naming (as, if $M = M' + \frac{1}{2}$, then $M'$ is a half-integer, so $N$ is even, and so, $u_1 = t_{min}$ to make 
$
\sum \limits _{j = 1} ^{Q-1} {d_j
\left(
\begin{matrix}
u_j \\
\cdots \\
u_j^N
\end{matrix}
\right)}
$ 
an $\left(\frac{N+1}{2}\right)$-naming, and in a naming, any terms involving the point at $t_{min}$ are automatically the first terms of the sum). Thus, the required $\left(M'-1\right)$-naming exists.
\end{proof}
\begin{lem} [Naming Lemma]
\label{naminglem}
Every point in the convex hull of $C_N$ has an $\left(\frac{N+1}{2}\right)$-naming.
\end{lem}
\begin{proof}
Every point in the convex hull of $C_N$ has an $M$-naming for some positive integer $M$. This naming can be extended to an $\left(M + \frac{1}{2}\right)$-naming (of that point) by adding a term with coefficient zero and the point at $t_{min}$ at the beginning, to obtain a half-integer naming. Beginning with either this integer naming or this half-integer naming (to match whether $\frac{N+1}{2}$ is an integer or a half-integer), if it is an $M'$-naming with $M' > \frac{N+1}{2}$, then Lemma \ref{gendimreductionlem}, the Generalized Dimension Reduction Lemma, can be repeatedly used to, eventually, reduce the naming to an $\left(\frac{N+1}{2}\right)$-naming (still of the same point as before). If it is an $M'$-naming with $M' < \frac{N+1}{2}$, then it can be padded with extra terms with zero coefficients to become an $\left(\frac{N+1}{2}\right)$-naming (still of the same point as before).
\end{proof}
\begin{thme}
Every point in the convex hull of the curve that is the image of the parametric function $C_N: \left[t_{min}, t_{max}\right] \rightarrow \Re^N$, such that
$C_N\left(t\right) = 
\left(
\begin{matrix}
t \\
t^2 \\
t^3 \\
\cdots \\
t^N
\end{matrix}
\right)
$, can be represented as a convex combination of at most $\frac{N+1}{2}$ points on this curve if $N$ is odd, or as a convex combination of at most $\frac{N+2}{2}$ points on this curve if $N$ is even. Furthermore, if $N$ is even, one of these at most $\frac{N+2}{2}$ points on the curve can be required to be the point $C_N\left(t_{min}\right)$.
\end{thme}
\begin{proof}
The $\left(\frac{N+1}{2}\right)$-naming required to exist by Lemma \ref{naminglem}, the Naming Lemma, is such. If $N$ is odd, then that is a convex combination of at most $\frac{N+1}{2}$ points on $C_N$. If $N$ is even, then this $\left(\frac{N+1}{2}\right)$-naming is an $\left(\frac{N+2}{2}\right)$-naming, whose first point is at $t_{min}$, so this naming is a convex combination of at most $\frac{N+2}{2}$, one of which is $C_N\left(t_{min}\right)$.
\end{proof}
\section{Uniqueness}
\label{uniquenesssection}
\begin{lem}[At Most One Non-Reducible Naming Lemma]
\label{atmostonenaminglemma}
Every point in the convex hull of $C_N$ has at most one non-reducible $M'$-naming, where $\frac{N+1}{2} - M'$ is a nonnegative integer. There is at most one such naming, in total, for all these possible values of $M'$.
\end{lem}
\begin{proof}
Suppose that there is a point 
$
\left(
\begin{matrix}
v_1 \\
\cdots \\
v_n \\
\end{matrix}
\right)
$ 
in the convex hull of $C_N$, for which there are at least two such namings $P_1$ and $P_2$, possibly with different choices of $M'$, so that $P_1$ is an $M'_1$-naming, and $P_2$ is an $M'_2$ naming.

Let $q$ be $1$ or $2$. Let $M_q$ be $M'_q$ or $M'_q + \frac{1}{2}$, whichever of these is an integer. That makes $P_q$ an $M_q$-naming, either because $M'_q$ is an integer and equals $M_q$, or because $M'_q$ is a half-integer, in which case, an $M'_q$-naming is an $M_q$-naming (as $M'_q + \frac{1}{2} = M_q$ in this case).

Let 
$
P_q = \sum \limits _{j = 1} ^{M_q} {c_{q, j}
\left(
\begin{matrix}
t_{q, j} \\
t_{q, j}^2 \\
t_{q, j}^3 \\
\cdots \\
t_{q, j}^N
\end{matrix}
\right)
}
$. 
If $N$ is odd, then, since $\frac{N+1}{2}-M'_q$ is a nonnegative integer, it follows that $\frac{N+1}{2}-M_q$ is a nonnegative integer, so $\frac{N+1}{2}-M_q \geq 0$, and so, $M_q \leq \frac{N+1}{2}$. If, instead, $N$ is even, then, since $\frac{N+1}{2}-M'_q$ is a nonnegative integer, it follows that $\frac{N+1}{2}-\left(M_q-\frac{1}{2}\right)$ is a nonnegative integer, so $\frac{N+1}{2}-\left(M_q-\frac{1}{2}\right) \geq 0$, and so, $M_q \leq \frac{N+2}{2}$. Additionally, if $N$ is even, then, again, since $\frac{N+1}{2}-M'_q$ is a nonnegative integer, $M'_q$ is a half-integer, so that means that $t_{q, 1} = t_{min}$. This is summarized in the following table, which holds for both $q = 1$ and $q = 2$:

\begin{tabular}{|c|c|c|c|c|}
\hline 
$N$ & $M'_q$ & Condition & Extra constraint \\ 
\hline 
odd & integer & $M_q \leq \frac{N+1}{2}$ & none \\ 
\hline 
even & half-integer & $M_q \leq \frac{N+2}{2}$ & $t_{q, 1} = t_{min}$ \\ 
\hline 
\end{tabular}

It follows that 
\[
\sum \limits _{j = 1} ^{M_1} {c_{1, j}
\left(
\begin{matrix}
1 \\
t_{1, j} \\
\cdots \\
t_{1, j}^N
\end{matrix}
\right)
}
=
\sum \limits _{j = 1} ^{M_2} {c_{2, j}
\left(
\begin{matrix}
1 \\
t_{2, j} \\
\cdots \\
t_{2, j}^N
\end{matrix}
\right)
}
\]
with the equality in the first components being because the coefficients of a naming sum to $1$, and with the equality in the other components being because $P_1$ and $P_2$ are namings of the same point. This equation can be simplified by moving the right side to the left side, and combining any like terms (ones with the same vector), and removing any resulting zero terms, leaving something of the form
\[
\sum \limits _{j = 1} ^{Q} {d_j
\left(
\begin{matrix}
1 \\
u_j \\
\cdots \\
u_j^N
\end{matrix}
\right)
}
=
\left(
\begin{matrix}
0 \\
0 \\
\cdots \\
0
\end{matrix}
\right)
\]
If this equation is $0 = 0$, then, as the $t_{1, j}$ are pairwise distinct and the $t_{2, j}$ are pairwise distinct (or else, one of $P_1$ and $P_2$ is reducible), the only simplifications possible are the merger of one term on the left side and one term on the right side. This means that both $P_1$ and $P_2$ are series of exactly the same terms (each 
$
c_{1, j}
\left(
\begin{matrix}
1 \\
t_{1, j} \\
\cdots \\
t_{1, j}^N
\end{matrix}
\right)
$
term cancelling out a 
$
c_{1, j'}
\left(
\begin{matrix}
1 \\
t_{1, j'} \\
\cdots \\
t_{1, j'}^N
\end{matrix}
\right)
$
term means that $c_j = c_j'$ and that $t_j = t_j'$), and they are in exactly the same order (strictly ascending, because the $t_{1, j}$ are pairwise distinct and the $t_{2, j}$ are pairwise distinct), so $P_1$ and $P_2$ are exactly the same naming. Thus, only the case that at least one term remains, which is the $Q \geq 1$ case, needs to be examined.

$Q$ is at most $M_1 + M_2$, which is at most $N+1$ if $N$ is odd, or at most $N+2$ if $N$ is even. However, if $N$ is even, then at least one pair of terms was merged, because $t_{1, 1} = t_{min} = t_{2, 1}$ in this case. That means that $Q \leq N+1$. Thus, the equation can be truncated to
\[
\sum \limits _{j = 1} ^{Q} {d_j
\left(
\begin{matrix}
1 \\
u_j \\
\cdots \\
u_j^{Q-1}
\end{matrix}
\right)
}
=
\left(
\begin{matrix}
0 \\
0 \\
\cdots \\
0
\end{matrix}
\right)
\]
This can be written in matrix form, as
\[
\left(
\begin{matrix}
1 & \cdots & 1 \\
u_1 & \cdots & u_Q \\
\cdots & \cdots & \cdots \\
u_1^{Q-1} & \cdots & u_Q^{Q-1} \\
\end{matrix}
\right)
\left(
\begin{matrix}
d_1 \\
\cdots \\
d_Q \\
\end{matrix}
\right)
= 
\left(
\begin{matrix}
0 \\
0 \\
\cdots \\
0 \\
\end{matrix}
\right)
\]
The matrix is a pseudo-Vandermonde matrix (in fact, a Vandermonde matrix), with its variables distinct (or further simplification coud have been made), so it is invertible. Thus, 
$
\left(
\begin{matrix}
d_1 \\
\cdots \\
d_Q \\
\end{matrix}
\right)
$ 
has only zero entries, but that cannot be true, because those terms would have been removed during simplification. Thus, since 
$
\left(
\begin{matrix}
d_1 \\
\cdots \\
d_Q \\
\end{matrix}
\right)
$ 
has all zero entries and has no zero entries, that vector has no entries at all, but that case was already eliminated. Thus, there is a contradiction, so $P_1$ and $P_2$ are the same naming.
\end{proof}
\begin{lem}[Nonreducibility of Low-Size Namings Lemma]
\label{nonreducibilitylowlemma}
Every $\frac{1}{2}$-naming of a point in the convex hull of $C_N$ is nonreducible, and every $1$-naming of a point in the convex hull of $C_N$ is nonreducible.
\end{lem}
\begin{proof}
Every $\frac{1}{2}$-naming of a point is a $1$-naming of that point. The single coefficient, which is the sum of the coefficients, must be $1$, which is not $0$. There are no two adjacent terms, so there can be no two adjacent terms with equal $t$s. Thus, the naming is not reducible.
\end{proof}
\begin{lem}[Unique Non-Reducible Naming Lemma]
\label{uniquenaminglemma}
Every point in the convex hull of $C_N$ has at most one non-reducible $M'$-naming, where $\frac{N+1}{2} - M'$ is a nonnegative integer. There is exactly one such naming, in total, for all these possible values of $M'$.
\end{lem}
\begin{proof}
Let 
$
\left(
\begin{matrix}
v_1 \\
\cdots \\
v_N\\
\end{matrix}
\right)
$ 
be a point in the convex hull of $C_N$. By Lemma \ref{naminglem}, the Naming Lemma, this point has an $\left(\frac{N+1}{2}\right)$-naming. Then, let $M'$ be a positive integer or a positive half-integer, for which $\frac{N+1}{2} - M'$ is a nonnegative integer, and for which, there is at least one $M'$-naming of 
$
\left(
\begin{matrix}
v_1 \\
\cdots \\
v_N\\
\end{matrix}
\right)
$. 
($M'$ exists, because
$
\left(
\begin{matrix}
v_1 \\
\cdots \\
v_N\\
\end{matrix}
\right)
$ 
has an $\left(\frac{N+1}{2}\right)$-naming, and $\frac{N+1}{2} - \frac{N+1}{2} = 0$, which is a nonnegative integer.) If there are multiple possible values for $M'$, let $M'$ take the lowest possible value. Let $P$ be an $M'$-naming of 
$
\left(
\begin{matrix}
v_1 \\
\cdots \\
v_N\\
\end{matrix}
\right)
$. 
$P$ is non-reducible, either because $M' = \frac{1}{2}$ or $M' = 1$, in which case Lemma \ref{nonreducibilitylowlemma}, the Nonreducibility of Low-Size Namings Lemma, applies, or otherwise, because if $P$ were reducible, then, by Lemma \ref{dimreductionlem}, the Dimension Reduction Lemma, 
$
\left(
\begin{matrix}
v_1 \\
\cdots \\
v_N\\
\end{matrix}
\right)
$ 
would have an $\left(M' - 1\right)$-naming, and, furthermore, $\frac{N+1}{2} - \left(M' - 1\right)$ would still be a nonnegative integer, meaning that $M'$ was not set at its lowest possible value (which contradicts the stipulation that $M'$ be set at its lowest possible value).

That provides a nonreducible $M'$-naming, with $\frac{N+1}{2} - M'$ a nonnegative integer. By Lemma \ref{atmostonenaminglemma}, the At Most One Non-Reducible Naming Lemma, it is the only one, for all such values of $M'$, combined.
\end{proof}
\begin{defn}
Let $N$ be a positive integer, and let $M'$ be a positive integer or a positive half-integer. An $M'$-naming of a point in the convex hull of $C_N$ is a \emph{correct-parity-for-$N$} naming if $M' - \frac{N+1}{2}$ is an integer. $M'$ is itself also \emph{of the correct parity for $N$} if $M' - \frac{N+1}{2}$ is an integer.
\end{defn}
\begin{rmk}
By this definition, if $M'$ is of the correct parity for $N$, then an $M'$-naming of a point in the convex hull of $C_N$ is a \emph{correct-parity-for-$N$} naming. Also by this definition, a correct-parity-for-$N$ naming of a point in the convex hull of $C_N$ is an $M'$-naming for some value of $M'$ of the correct parity for $N$.
\end{rmk}
\begin{defn}
A zero-addition of an $M'$-naming of a point in the convex hull of $C_N$, with $M'$ of the correct parity for $N$, is the same naming with an extra term with coefficient zero.
\end{defn}
\begin{defn}
An equal-split of an $M'$-naming of a point in the convex hull of $C_N$, with $M'$ of the correct parity for $N$, is the same naming with an one of its terms split into two terms that use the same point on $C_N$, and whose coefficients sum to the coefficient of the term that was split.
\end{defn}
\begin{defn}
A direct simplification of an $M'$-naming of a point in the convex hull of $C_N$, with $M'$ of the correct parity for $N$, is the same naming, except either with a single term with coefficient zero removed, or with two terms using the same point on $C_N$ merged into a single term; in other words, the other naming is a zero-addition or an equal-split of the original naming.
\end{defn}
\begin{defn}[Definition of Direct Equivalence]
\label{directequivalencedef}
Two correct-parity namings of points in the convex hull of $C_N$ are directly equivalent, if and only if, (at least) one of them is a direct simplification of the other.
\end{defn}
\begin{defn}[Definition of Equivalence]
\label{equivalencedef}
Two correct-parity-for-$N$ namings $P_1$ and $P_2$ are equivalent, if there is a sequence of correct-parity-for-$N$ namings $P_1, P'_1, P'_2, \cdots, P'_A, P_2$, such that any two adjacent entries in the sequence are directly equivalent to each other.
\end{defn}
\begin{lem}
If $N$ is a positive integer, then, equivalence, as defined in Definition \ref{equivalencedef}, the Definition of Equivalence, is, in fact, an equivalence relation on the set of correct-parity-for-$N$ namings of points in the convex hull of $C_N$. 
\end{lem}
\begin{proof}
It suffices to show that equivalence is a reflexive, symmetric, and transitive relation. Let $P_1$, $P_2$, and $P_3$ be correct-parity-for-$N$ namings of points in the convex hull of $C_N$.

$P_1$ is equivalent to $P_1$, because there is a sequence of correct-parity-for-$N$ namings, $P_1$, such that any two adjacent entries in the sequence are directly equivalent to each other, there being no pairs of adjacent entries in this one-naming sequence.

If $P_1$ is equivalent to $P_2$, there exists a sequence of correct-parity-for-$N$ namings $P_1, P'_1, P'_2, \cdots, P'_A, P_2$, such that any two adjacent entries in the sequence are directly equivalent to each other. The same also holds for the reversed sequence $P_2, P'_A, P'_{A-1}, \cdots, P'_1, P_1$, by symmetry of direct equivalence (which was shown in the remark after Definition \ref{directequivalencedef}, the Definition of Direct Equivalence). That makes $P_2$ equivalent to $P_1$.

If $P_1$ is equivalent to $P_2$, and $P_2$ is equivalent to $P_3$, then there exist sequences of correct-parity-for-$N$ namings $P_1, P'_1, P'_2, \cdots, P'_A, P_2$ and $P_2, P'_{A+1}, P'_{A+2}, \cdots, P'_B, P_3$, such that, in both seqences, any two adjacent entries are directly equivalent to each other. For the joined sequence $P_1, P'_1, P'_2, \cdots, P'_A, P_2, P'_{A+1}, P'_{A+2}, \cdots, P'_B, P_3$, any two adjacent entries in the sequence are directly equivalent to each other. ($P_2$ and $P'_A$ are directly equivalent to each other, because $P'_A, P_2$ is part of the sequence showing that $P_1$ is equivalent to $P_2$, while $P_2$ and $P_{A+1}$ are directly equivalent to each other, because $P_2, P'_{A+1}$ is part of the sequence showing that $P_2$ is equivalent to $P_3$.) Thus, $P_1$ is equivalent to $P_3$.

Therefore, equivalence is, in fact, an equivalence relation on the set of correct-parity-for-$N$ namings of points in the convex hull of $C_N$. 
\end{proof}
\begin{lem}[Lemma on Directly-Equivalent Namings]
\label{direquivlemma}
If $N$ is a positive integer, then two directly-equivalent, correct-parity-for-$N$, namings of points in the convex hull of $C_N$ are namings of the same point.
\end{lem}
\begin{proof}
For two directly-equivalent namings, correct-parity-for-$N$, namings of points in the convex hull of $C_N$, one of them is a direct simplification of the other, and therefore, one is a zero-addition or an equal-split of the other. Neither adding a zero term, nor splitting a term into two, so that the two coefficients sum to the original coefficient, affects the result of the naming when evaluated. Thus, the two namings are namings of the same point.
\end{proof}
\begin{lem}[Lemma on Equivalent Namings]
\label{equivlemma}
If $N$ is a positive integer, then two equivalent, correct-parity-for-$N$, namings of points in the convex hull of $C_N$ are namings of the same point.
\end{lem}
\begin{proof}
Let $P_1$ be correct-parity-for-$N$ a naming of a point
$
\left(
\begin{matrix}
v_{1, 1} \\
\cdots \\
v_{1, N} \\
\end{matrix}
\right)
$ 
in the convex hull of $C_N$, and let $P_2$ be a correct-parity-for-$N$ naming of another point in the convex hull of $C_N$, such that $P_1$ and $P_2$ are equivalent. Therefore, there is a sequence of correct-parity-for-$N$ namings $P_1, P'_1, P'_2, \cdots, P'_A, P_2$, such that any two adjacent entries in the sequence are directly equivalent to each other. 

If this sequence has only one entry, then $P_1$ and $P_2$ are the same naming, and therefore are namings of the same point. If this sequence has only two entries, then $P_1$ and $P_2$ are directly equivalent, and therefore, by Lemma \ref{direquivlemma}, the Lemma on Directly-Equivalent Namings, are namings of the same point. 

If this sequence has three or more entries, then any two adjacent entries in the sequence are directly equivalent, and therefore, by Lemma \ref{direquivlemma}, are namings of the same point. Thus, $P_1$ and $P'_1$ are namings of the same point, as are $P'_1$ and $P'_2$, and so on, and as are $P'_A$ and $P_2$, so all the namings in the sequence are namings of the same point. In particular, $P_1$ and $P_2$ are namings of the same point.
\end{proof}

\begin{defn}
Let $N$ be a positive integer, and let $M'$ be a positive integer or a positive half-integer. An $M'$-naming of a point in the convex hull of $C_N$ is a \emph{proper-for-$N$} naming if it is a correct-parity-for-$N$ naming and if $M' \leq \frac{N+1}{2}$.
\end{defn}
\begin{rmk}
This definition is motivated by Lemma \ref{uniquenaminglemma}, the Unique Non-Reducible Naming Lemma, which, in these terms, states that every point in the convex hull of $C_N$ has exactly one proper-for-$N$, non-reducible, naming.
\end{rmk}

\begin{defn}
Let $N$ be a positive integer, and let $P$ be a proper-for-$N$ naming of a point in the convex hull of $C_N$. Then, a \emph{canonical form} of $P$ is a proper-for-$N$, non-reducible, naming that is equivalent to $P$.
\end{defn}

\begin{lem}[Reduction of Reducible Namings Lemma]
\label{reductionreduciblelemma}
Every reducible, proper-for-$N$, naming of a point in the convex hull of $C_N$, is either a zero-addition or an equal-split of a proper-for-$N$ naming of a point in the convex hull of $C_N$.
\end{lem}
\begin{rmk}
Both namings are namings of the same point, by Lemma \ref{equivlemma}, the Lemma on Equivalent Namings.
\end{rmk}
\begin{proof}
If a proper-for-$N$ naming of a point in $C_N$ is reducible, then either one of its terms (other than the first one if $N$ is even) has coefficient zero, or two of its adjacent terms use the same point. Remove the zero term, or merge those two adjacent terms. The original naming is either a zero-addition or an equal-split of the new naming (and the new naming is of the correct parity for $N$, because, if $N$ is even, then the first term was not the one removed). The new naming is proper for $N$, because if the original naming was an $M'$-naming, then the new naming is an $\left(M' - 1\right)$-naming (and because $M' - 1 \leq \frac{N+1}{2}$ follows from $M' \leq \frac{N+1}{2}$). 
\end{proof}

\begin{lem} [Unique Canonical Form Lemma]
\label{uniquecanonicallemma}
Every proper-for-$N$ naming of a point in the convex hull of $C_N$ has exactly one canonical form.
\end{lem}
\begin{proof}
If a naming has two distinct canonical forms, then that naming is equivalent to both canonical forms, which are thus equivalent to each other. Therefore, by Lemma \ref{equivlemma}, the Lemma on Equivalent Namings, both canonical forms are namings of the same point. Both namings are proper-for-$N$ namings and both namings are namings of the same point, and therefore, by Lemma \ref{atmostonenaminglemma}, the At Most One Non-Reducible Naming Lemma, they are the same naming. That shows that every proper-for-$N$ naming of a point in the convex hull of $C_N$ has at most one canonical form.

If a proper-for-$N$ naming $P$ is not reducible, then it is its own canonical form. Otherwise, it Lemma \ref{reductionreduciblelemma}, Reduction of Reducible Namings Lemma, states that $P$ is either a zero-addition of some proper-for-$N$ naming $P'_1$, which, if reducible, is itself either a zero-addition of some proper-for-$N$ naming $P'_2$, and so on until a non-reducible naming is reached (this does not go on forever, as each step reduces the number of terms in the naming by $1$. Each step preserves equivalence to the original naming $P$, so the last, non-reducible, naming is a canonical form of $P$. Thus, every proper-for-$N$ naming has at least one canonical form.

Therefore, every proper-for-$N$ naming of a point in the convex hull of $C_N$ has exactly one canonical form.
\end{proof}
\begin{rmk}
The preceding lemma means that it is unambiguous to mention \emph{the} canonical form of a proper-for-$N$ naming of a point in the convex hull of $C_N$, as that naming has exactly one canonical form.
\end{rmk}

\begin{lem} [Equivalence Lemma]
\label{equivalencelemma}
Let $N$ be an integer. Two proper-for-$N$ namings of the same point in the convex hull of $C_N$ are equivalent.
\end{lem}
\begin{proof}
The two namings are equivalent to their canonical forms (which are unique by Lemma \ref{uniquecanonicallemma}, the Unique Canonical Form Lemma), which are thus still namings of the same point, as equivalent namings are namings of the same point by Lemma \ref{equivlemma}, the Lemma on Equivalent Namings. The two canonical forms are both proper-for-$N$ namings of the same point in the convex hull of $C_N$, and they are both non-reducible, so they are the same naming by Lemma \ref{atmostonenaminglemma}, the At Most One Non-Reducible Naming Lemma. Thus, both namings have the same canonical form, and thus, both namings are equivalent to the same naming (the canonical form), and are thus equivalent to each other.
\end{proof}

\begin{lem}
For every point in the convex hull of the curve that is the image of the parametric function $C_N: \left[t_{min}, t_{max}\right] \rightarrow \Re^N$, such that
$C_N\left(t\right) = 
\left(
\begin{matrix}
t \\
t^2 \\
t^3 \\
\cdots \\
t^N
\end{matrix}
\right)
$, and for any two representations as a convex combination of at most $\frac{N+1}{2}$ points on this curve if $N$ is odd, or as a convex combination of $\frac{N+2}{2}$ points on this curve if $N$ is even, where, furthermore, if $N$ is even, one of these at most $\frac{N+2}{2}$ points on the curve is required to be the point $C_N\left(t_{min}\right)$, it is possible to obtain one representation from the other by a sequence of three kinds of steps:

1) exchanging two terms with each other,

2) adding or removing a term with coefficient zero, or

3) combining two terms using the same point on the curve into a single term or splitting one term into two.
\end{lem}

\begin{proof}
The terms of the two representations of a point can be rearranged, so that their $t$s are in ascending order, which only uses step 1). The result is two namings, $P_1$ and $P_2$, of that point. Let $q$ be $1$ or $2$. Let $P_q$ be an $M_q$-naming, with $M_1$ a positive integer. Let $M'_q$ be $M_q$ if $N$ is odd, or let $M'_q$ be $M_q - \frac{1}{2}$ if $N$ is even. Then, $P_q$ is an $M'_q$-naming (either because $N$ is odd, and therefore, $M'_q = M_q$, or because $N$ is even, and therefore, $M'_q = M_q - \frac{1}{2}$ and the first term of $P_2$ uses the point at $t_{min}$, because this was one of the points used in the convex combination, which was moved to the front of the list when the convex combination was turned into a naming). $M'_q$ is of the correct parity for $N$, because $M' - \frac{N+1}{2}$ is an integer, regardless of whether $N$ is odd or even. $P_q$ is thus a correct-parity-for-$N$ naming. $P_q$ is also a proper-for-$N$ naming, because, if $N$ is odd, then $M'_q = M_q \leq \frac{N+1}{2}$, while if $N$ is even, $M'_q = M_q - \frac{1}{2} \leq \frac{N+2}{2} - \frac{1}{2} = \frac{N+1}{2}$.

By Lemma \ref{equivalencelemma}, the Equivalence Lemma, $P_1$ and $P_2$ are equvalent, and therefore, one of them can be reached from the other by a sequence of zero-additions, equal-splits, and the reverses of these steps, all of which are steps of kind 2) or 3). 

Thus, from the first representation, a sequence of steps of kind 1) reach $P_1$, from which a sequence of steps of kinds 2) and 3) reach $P_2$, from which a sequence of steps of kind 1) reach the second representation (the reverse of the sequence needed to reach $P_2$ from the second representation). Joining these sequences provides a way to obtain the second representation from the first by a sequence of steps of kinds 1), 2), and 3).
\end{proof}

\section{Homeomorphism}
\label{homeomorphismsection}

\begin{defn}
For any positive integer $N$, $Nam_N$ is the set of all $\left(\frac{N+1}{2}\right)$-namings of points in the convex hull of $C_N$.
\end{defn}
\begin{lem} [Alternative Definition of $Nam_N$]
\label{altdefinitionnamnlemma}
Let $N$ be a positive integer, and let $M = \frac{N+1}{2}$ if $N$ is odd, or let $M = \frac{N+2}{2}$ if $N$ is even. $Nam_N$ is exactly the set of all tuples $\left(c_1, \cdots, c_M, t_1, \cdots, t_M\right)$ that satisfy all of the following properties:
\begin{align*}
& \forall_{ j \in \left \lbrace 1, 2, \cdots, M \right \rbrace} t_{min} \leq t_j \leq t_{max} \\
& \forall_{ j \in \left \lbrace 1, 2, \cdots, M-1 \right \rbrace} t_j \leq t_{j+1} \\
& \forall _{j \in \left \lbrace 1, 2, \cdots, M \right \rbrace} c_j \geq 0 \\
& \sum \limits _{j = 1} ^{M} {c_j} = 1 \\
\end{align*}
and the additional property that $t_1 = t_{min}$ if $N$ is even.
\end{lem}
\begin{proof}
When taken together, all of
\begin{align*}
& \forall_{ j \in \left \lbrace 1, 2, \cdots, M \right \rbrace} t_{min} \leq t_j \leq t_{max} \\
& \forall _{j \in \left \lbrace 1, 2, \cdots, M \right \rbrace} c_j \geq 0 \\
& \sum \limits _{j = 1} ^{M} {c_j} = 1 \\
\end{align*} 
are necessary and sufficient conditions required for 
$
\sum \limits _{j = 1} ^{M} {c_j
\left(
\begin{matrix}
t_j \\
\cdots \\
t_j^N \\
\end{matrix}
\right)
}
$ 
to be a convex combination of points on $C_N$, while $\forall_{ j \in \left \lbrace 1, 2, \cdots, M-1 \right \rbrace} t_j \leq t_{j+1}$ guarantees that the terms are in ascending order of $t$, and this is necessary and sufficient for 
$
\sum \limits _{j = 1} ^{M} {c_j
\left(
\begin{matrix}
t_j \\
\cdots \\
t_j^N \\
\end{matrix}
\right)
}
$
(or its other representation, $\left(c_1, \cdots, c_M, t_1, \cdots, t_M\right)$ to be an $M$-naming. When $N$ is odd, these are exactly the $\left(\frac{N+1}{2}\right)$-namings, because $M = \frac{N+1}{2}$ in this case. When $N$ is even, these are exactly the $\left(\frac{N+2}{2}\right)$-namings, because $M = \frac{N+1}{2}$ in this case, and among them, the $\left(\frac{N+2}{2}\right)$-namings are exactly the ones that satisfy the additional property that $t_1 = t_{min}$. Thus, whether $N$ is odd or even, the set of tuples meeting this list of properties is exactly $Nam_N$.
\end{proof}

\begin{defn}[Definition of Equivalence of $\left(\frac{N+1}{2}\right)$-namings]
\label{equivN+1over2def}
For any positive integer $N$, two $\left(\frac{N+1}{2}\right)$-namings of points in the convex hull of $C_N$ are \emph{equivalent}, if and only if they are equivalent as correct-parity-for-$N$ namings of points in the convex hull of $C_N$; that is, if and only if they are equivalent according to Definition \ref{equivalencedef}, the Definition of Equivalence. If $P_1$ and $P_2$ are the two namings, then this equivalence is denoted by $P_1 \cong P_2$.
\end{defn}
\begin{rmk}
The sequence of correct-parity-for-$N$ namings that Definition \ref{equivalencedef} requires to exist is allowed to have, as some of its entries, namings that are not $\left(\frac{N+1}{2}\right)$-namings.
\end{rmk}
\begin{rmk}
Two equivalent namings are namings of the same point, as shown in Lemma \ref{equivlemma}, the Lemma on Equivalent Namings.
\end{rmk}

\begin{defn}
$Nam_N/{\cong}$ is the set of all equivalence classes of $\left(\frac{N+1}{2}\right)$-namings of points in the convex hull of $C_N$ (that is, of elements of $Nam_N$), where any two namings equivalent under Definition \ref{equivN+1over2def}, the Definition of Equivalence of $\left(\frac{N+1}{2}\right)$-namings, are in the same equivalence class, while any two namings not equivalent under Definition \ref{equivN+1over2def} are not in the same equivalence class.
\end{defn}

\begin{defn}
$conv\left(C_N\right)$ is the convex hull of $C_N$.
\end{defn}

\begin{defn}
Let $N$ be a positive integer. The usual topology for $\Re^N$ is the Euclidean topology for $\Re^N$. Likewise, the usual topology for $\Re^{2M}$, where $M$ is $\frac{N+1}{2}$ if $N$ is odd, or $\frac{N+2}{2}$ if $N$ is even, is the Euclidean topology for $\Re^{2M}$.
\end{defn}

\begin{defn}
Let $N$ be a positive integer. The \emph{usual topology} for $Nam_N$ is defined to be the topology obtained by treating $Nam_N$ as a subspace of the topological space $\Re^{2M}$, where $M$ is $\frac{N+1}{2}$ if $N$ is odd, or $\frac{N+2}{2}$ if $N$ is even. ($M$ is the number of entries in each tuple in $Nam_N$.) That is, a set $S_{Nam}$ is open in $Nam_N$ if, and only if, there exists a set $S_{\Re, 2M}$ that is open in $\Re^{2M}$, and that additionally satisfies $S_{\Re, 2M} \cap Nam_N = S_{Nam}$.
\end{defn}

\begin{defn}
Let $N$ be a positive integer. The \emph{usual topology} for $Nam_N/{\cong}$ is defined to be the topology obtained by treating $Nam_N/{\cong}$ as a quotient of the topological space $Nam_N$ (which has its usual topology). That is, a set $S_{QNam}$ is open in $Nam_N/{\cong}$ if, and only if, the union of all the members of $S_{QNam}$ (which can be treated as sets, because they are equivalence classes) is open in $Nam_N$.
\end{defn}

\begin{defn}
Let $N$ be a positive integer. The \emph{usual topology} for $conv\left(C_N\right)$ is defined to be the topology obtained by treating $conv\left(C_N\right)$ as a subspace of the topological space $\Re^N$. That is, a set $S_{conv}$ is open in $conv\left(C_N\right)$ if, and only if, there exists a set $S_{\Re, N}$ that is open in $\Re^N$, and that additionally satisfies $S_{\Re} \cap conv\left(C_N\right) = S_{conv}$.
\end{defn}

\begin{defn}
Let $N$ be a positive integer. Then, $F_{\Re}$ is defined to be the function from $\Re^{2M}$ to $Re^{N}$, where $M$ is $\frac{N+1}{2}$ if $N$ is odd, and $M$ is $\frac{N+2}{2}$ if $N$ is even, such that
\[
F_{\Re}\left(c_1, \cdots, c_M, t_1, \cdots, t_M\right) =
\sum \limits _{j = 1} ^{M} {c_j
\left(
\begin{matrix}
t_j \\
\cdots \\
t_j^N \\
\end{matrix}
\right)
}
\]
\end{defn}

\begin{defn}
Let $N$ be a positive integer. Then, $F_{Nam}$ is defined to be the function from $Nam_N$ to $conv\left(C_N\right)$, with $F_{Nam}\left(P\right) = F_{\Re}\left(P\right)$, where $P$ is treated as an element of $\Re^{2M}$, where $M$ is $\frac{N+1}{2}$ if $N$ is odd, and $M$ is $\frac{N+2}{2}$ if $N$ is even, on the right side.
\end{defn}
\begin{rmk}
By this definition, $F_{Nam}$ is the function that evaluates the naming, when taken as a convex combination, so applying it to a naming always results in a point in $conv(C_N)$.
\end{rmk}
\begin{defn}
Let $N$ be a positive integer. Then, $F_{QNam}$ is defined to be the function from $Nam_N/{\cong}$ to $conv\left(C_N\right)$, with $F_{QNam}\left(\tilde{P}\right) = F_{Nam}\left(P\right)$, where $P$ is a member of the equivalence class $\tilde{P}$.
\end{defn}
\begin{rmk}
If $P_1$ and $P_2$ are equivalent namings, then, by Lemma \ref{equivlemma}, the Lemma on Equivalent Namings, $P_1$ and $P_2$ are namings of the same point. For $q = 1$ and $q = 2$, if $P_q = \left(c_{q, 1}, \cdots, c_{q, M}, t_{q, 1}, \cdots, t_{q, M}\right)$ (where $M$ is $\frac{N+1}{2}$ if $N$ is odd, or $\frac{N+2}{2}$ if $N$ is even), then $P_q$ is a naming of and only of 
$
\sum \limits _{j = 1} ^{M} {c_{q, j}
\left(
\begin{matrix}
t_{q, j} \\
\cdots \\
t_{q, j}^N
\end{matrix}
\right)
}
$. 
Thus, $P_q$ is a naming of and only of $F_{Nam}\left(P_q\right)$, which makes $F_{Nam}\left(P_1\right)$ and $F_{Nam}\left(P_2\right)$ the same point. Thus, the definition of $F_{QNam}\left(\tilde{P}\right)$ as $F_{Nam}\left(P\right)$ is independent of the choice of representative of $\tilde{P}$.
\end{rmk}

\begin{lem} [Continuity Lemma]
\label{continuitylemma}
Let $N$ be a positive integer, and let $\Re^{2M}$, $Nam_N$, $Nam_N/{\cong}$, $\Re^N$, and $conv\left(C_N\right)$ have their usual topologies. Then, $F_{\Re}$, $F_{Nam}$, and $F_{QNam}$ are all continuous.
\end{lem}
\begin{proof}
Let $M$ be $\frac{N+1}{2}$ if $N$ is odd, or $\frac{N+2}{2}$ if $N$ is even.

$F_{\Re}$ is continuous, as each component of $F_{\Re}$ is a continuous function on its whole domain $\Re^{2M}$.

Let $O_{conv}$ be an open set in $conv\left(C_N\right)$. Then, $O_{conv} = O_{\Re} \cap conv\left(C_N\right)$ for some open set $O_{\Re}$ in $\Re^N$, because $conv\left(C_N\right)$ has the usual topology. Let $F^{-1} \left(S\right)$ (with whatever subscript) denote the inverse image of $S$ under $F$ (with the same subscript). Thus:
\begin{align*}
F_{Nam}^{-1}\left(O_{conv}\right) &= \left \lbrace P \in Nam_N | F_{Nam}\left(P\right) \in O_{conv} \right \rbrace \\
&= \left \lbrace P \in Nam_N | F_{Nam}\left(P\right)  \in \left(O_{\Re} \cap conv\left(C_N\right)\right) \right \rbrace \\
&= \left \lbrace P \in Nam_N | F_{Nam}\left(P\right) \in O_{\Re} \text{ and } F_{Nam}\left(P\right) \in conv\left(C_N\right) \right \rbrace \\
&= \left \lbrace P \in Nam_N | F_{Nam}\left(P\right) \in O_{\Re} \right \rbrace \text { (as } F_{Nam}\left(P\right) \in conv\left(C_N\right) \text{)} \\
&= \left \lbrace P \in \Re^{2M} | P \in Nam_N \text{ and } F_{Nam}\left(P\right) \in O_{\Re} \right \rbrace \\
&= \left \lbrace P \in \Re^{2M} | P \in Nam_N \right \rbrace \cap  \left \lbrace P \in \Re^{2M} | F_{Nam}\left(P\right) \in O_{\Re} \right \rbrace \\
&= Nam_N \cap \left \lbrace P \in \Re^{2M} | F_{Nam}\left(P\right) \in O_{\Re} \right \rbrace \\
&= Nam_N \cap \left \lbrace P \in \Re^{2M} | F_{\Re}\left(P\right) \in O_{\Re} \right \rbrace \text { (as } F_{Nam}\left(P\right) = F_{\Re}\left(P\right) \text{)} \\
&= Nam_N \cap F_{\Re}^{-1}\left(O_{\Re}\right)
\end{align*}
which is an open set in $Nam_N$, as it is the intersection of $Nam_N$ with the set $F_{\Re}^{-1}\left(O_{\Re}\right)$, which is open in $\Re^{2M}$, because it is the inverse image of a set $O_{\Re}$ that is open in $\Re^N$, under a continuous (as proven earlier) function $F_{\Re}$ from $\Re^{2M}$ to $\Re^N$. Thus, the inverse image of an open set in $conv\left(C_N\right)$ under $F_{Nam}$ is an open set in $Nam_N$, which makes $F_{Nam}$ continuous.

Let $O_{Qconv}$ be an open set in $conv\left(C_N\right)$. Then,
\[
F_{QNam}^{-1}\left(O_{Qconv}\right) = \left \lbrace \tilde{P} \in \left(Nam_N/{\cong}\right) | F_{QNam}\left(\tilde{P}\right) \in O_{Qconv} \right \rbrace
\]

This is an open set in $Nam_N/{\cong}$ if and only if the union of its elements is open in $Nam_N$. This union is:
\begin{align*}
\bigcup \limits _{\tilde{P} \in F_{QNam}^{-1}\left(O_{Qconv}\right)} {\left(\tilde{P}\right)} &= \left \lbrace P \in Nam_N | \exists _{\tilde{P} \in F_{QNam}^{-1}\left(O_{Qconv}\right)} P \in \tilde{P} \right \rbrace \\
&= \left \lbrace P \in Nam_N | \exists _{\tilde{P} \in \left(Nam_N/{\cong}\right)} \left(P \in \tilde{P} \text{ and } \tilde{P} \in F_{QNam}^{-1}\left(O_{Qconv}\right) \right) \right \rbrace \\
&= \left \lbrace P \in Nam_N | \exists _{\tilde{P} \in \left(Nam_N/{\cong}\right)} \left(P \in \tilde{P} \text{ and } F_{QNam}\left(\tilde{P}\right) \in O_{Qconv} \right) \right \rbrace \\
&= \left \lbrace P \in Nam_N | \exists _{\tilde{P} \in \left(Nam_N/{\cong}\right)} \left(P \in \tilde{P} \text{ and } F_{Nam}\left(P\right) \in O_{Qconv} \right) \right \rbrace \\
& \quad\quad \text{(because } F_{Nam}\left(P\right) = F_{QNam}\left(\tilde{P}\right) \text{ follows from } P \in \tilde{P} \text{)} \\
&= \left \lbrace P \in Nam_N | \left( \exists _{\tilde{P} \in \left(Nam_N/{\cong}\right)} P \in \tilde{P} \right) \text{ and } F_{Nam}\left(P\right) \in O_{Qconv} \right \rbrace \\
&= \left \lbrace P \in Nam_N | F_{Nam}\left(P\right) \in O_{Qconv} \right \rbrace \\
& \quad\quad \text{(because every } P \in Nam_N \text{ is in its equivalence class)} \\
&= F_{Nam}^{-1}\left(O_{Qconv}\right) \\
\end{align*}
which is an open set in $Nam_N/{\cong}$, as it is the inverse image of a set $O_{Qconv}$ that is open in $conv\left(C_N\right)$, under a continuous (as proven earlier) function $F_{Nam}$ from $Nam_N$ to $conv(C_N)$. Thus, the inverse image of an open set in $conv\left(C_N\right)$ under $F_{QNam}$ is an open set in $Nam_N/{\cong}$, which makes $F_{QNam}$ continuous.

Thus, $F_{\Re}$, $F_{Nam}$, and $F_{QNam}$ are all continuous if $\Re^{2M}$, $Nam_N$, $Nam_N/{\cong}$, $\Re^N$, and $conv\left(C_N\right)$ have their usual topologies.
\end{proof}

\begin{lem} [Compactness of $Nam_N$ Lemma]
\label{compactnessofnamn}
Let $N$ be a positive integer. Then, $Nam_N$ (with its usual topology) is compact.
\end{lem}
\begin{proof}
Let $M = \frac{N+1}{2}$ if $N$ is odd, or let $M = \frac{N+2}{2}$ if $N$ is even. By Lemma \ref{altdefinitionnamnlemma}, the Alternative Definition of $Nam_N$, $Nam_N$ is exactly the set of all tuples $\left(c_1, \cdots, c_M, t_1, \cdots, t_M\right)$ that satisfy all of the following properties:
\begin{align*}
& \forall_{ j \in \left \lbrace 1, 2, \cdots, M \right \rbrace} t_{min} \leq t_j \leq t_{max} \\
& \forall_{ j \in \left \lbrace 1, 2, \cdots, M-1 \right \rbrace} t_j \leq t_{j+1} \\
& \forall _{j \in \left \lbrace 1, 2, \cdots, M \right \rbrace} c_j \geq 0 \\
& \sum \limits _{j = 1} ^{M} {c_j} = 1 \\
\end{align*}
and the additional property that $t_1 = t_{min}$ if $N$ is even. The set satisfying any one of these properties (even the additional property, when $N$ is even) is closed, so $Nam_N$ is the intersection of closed sets, and is itself closed.

$Nam_N$ is bounded, because each $t_j$ is in the closed interval $\left[t_{min}, t_{max}\right]$, and because each $c_j$ is in the closed interval $\left[0, 1\right]$. (As all the $c_j$ are nonnegative, if any $c_j$ were greater than $1$, then $\sum \limits _{j = 1} ^{M} {c_j}$ would be greater than $1$.)

Since $Nam_N$ is closed and bounded, it is compact.
\end{proof}

\begin{lem} [Compactness of $Nam_N/{\cong}$ Lemma]
\label{compactnessofnamnovercong}
Let $N$ be a positive integer. Then, $Nam_N/{\cong}$ (with its usual topology) is compact.
\end{lem}
\begin{proof}
Let $I$ be an index set, and let $\left \lbrace O_{QNam, i} | i \in I \right \rbrace$ be an open cover of $Nam_N/{\cong}$. Thus, $\bigcup \limits _{i \in I} {O_{QNam, i}} = Nam_N/{\cong}$, and every $O_{QNam, i}$ is open in $Nam_N/{\cong}$.

For every $O_{QNam, i}$, since $O_{QNam, i}$ is open in $Nam_N/{\cong}$, it follows that $\bigcup \limits _{\tilde{P} \in O_{QNam, i}} {\left(\tilde{P}\right)}$ is open in $Nam_N$. Let this set be $O_{Nam, i}$. Then:

\begin{align*}
\bigcup \limits _{i \in I} {O_{Nam, i}} &= \bigcup \limits _{i \in I} {\left(\bigcup \limits _{\tilde{P} \in O_{QNam, i}} {\left(\tilde{P}\right)}\right)} \\
&= \bigcup \limits _{i \in I} {\left(\bigcup \limits _{\tilde{P} \in O_{QNam, i}} {\left(\left \lbrace P| P \in \tilde{P}\right \rbrace\right)}\right)} \\
&= \bigcup \limits _{i \in I} {\left(\left \lbrace P | \exists _{\tilde{P} \in O_{QNam, i}} P \in \tilde{P} \right \rbrace\right)} \\
&= \bigcup \limits _{i \in I} {\left(\left \lbrace P | \exists _{\tilde{P} \in Nam_N/{\cong}} \left(P \in \tilde{P} \text{ and } \tilde{P} \in O_{QNam, i}\right) \right \rbrace\right)} \\
&= \left \lbrace P |\exists_{i \in I} \left( \exists _{\tilde{P} \in Nam_N/{\cong}} \left(P \in \tilde{P} \text{ and } \tilde{P} \in O_{QNam, i}\right) \right) \right \rbrace \\
&= \left \lbrace P |\exists _{\tilde{P} \in Nam_N/{\cong}} {\left( \exists_{i \in I} \left(P \in \tilde{P} \text{ and } \tilde{P} \in O_{QNam, i}\right) \right)} \right \rbrace \\
&= \left \lbrace P |\exists _{\tilde{P} \in Nam_N/{\cong}} {\left( P \in \tilde{P} \text{ and } \exists_{i \in I} \left( \tilde{P} \in O_{QNam, i}\right) \right)} \right \rbrace \\
&= \left \lbrace P |\exists _{\tilde{P} \in Nam_N/{\cong}} {P \in \tilde{P}} \right \rbrace \text{ (since } \left \lbrace O_{QNam, i} | i \in I \right \rbrace \text{ is an open cover of } Nam_N/{\cong} \text{)} \\
&= Nam_N \text{ (since every } P \in Nam_N \text{ is in its equivalence class)} \\
\end{align*}

That makes $\left \lbrace O_{Nam, i} | i \in I \right \rbrace$ an open cover of $Nam_N$. By Lemma \ref{compactnessofnamn}, the Compactness of $Nam_N$ Lemma, $\left \lbrace O_{Nam, i} | i \in I \right \rbrace$ has a finite subcover. Let $\left \lbrace O_{Nam, i} | i \in I' \right \rbrace$ be this finite subcover, where $I'$ is a subset of $I$.

$\left \lbrace O_{QNam, i} | i \in I' \right \rbrace$ is a collection of open sets in $Nam_N/{\cong}$, as it is a subset of an open cover of $Nam_N/{\cong}$. Its union is $\bigcup \limits _{i \in I'} {O_{QNam, i}}$. Let $\tilde{P}$ be in $Nam_N/{\cong}$, and let $P$ be in $\tilde{P}$. $P$ is in some $O_{Nam, i}$ with $i \in I'$, because $\left \lbrace O_{Nam, I} | i \in I' \right \rbrace$ is an open cover of $Nam_N$. Therefore, for this $i$, $P \in \bigcup \limits _{\tilde{P} \in O_{QNam, i}} {\left(\tilde{P}\right)}$. Since $P$ is in one and only one equivalence class, and since the sets comprising $\bigcup \limits _{\tilde{P} \in O_{QNam, i}} {\left(\tilde{P}\right)}$ are equivalence classes, it follows that the $\tilde{P}$ that $P$ is in must satisfy $\tilde{P} \in O_{QNam, i}$. As $i \in I'$, this means that all equivalence classes are elements of $\bigcup \limits _{i \in I'} {O_{QNam, i}}$. That makes $\left \lbrace O_{QNam, i} | i \in I' \right \rbrace$ an open cover of $Nam_N/{\cong}$. Therefore, the open cover $\left \lbrace O_{QNam, i} | i \in I \right \rbrace$ of $Nam_N/{\cong}$ has a finite subcover, namely $\bigcup \limits _{i \in I'} {O_{QNam, i}}$.

Thus, every open cover of $Nam_N/{\cong}$ has a finite subcover. Therefore, $Nam_N/{\cong}$ is compact.

\end{proof}

\begin{lem} [Hausdorff Property of $conv\left(C_N\right)$ Lemma]
\label{hausdorffconvcnlemma}
Let $N$ be a positive integer. Then, $conv\left(C_N\right)$ (with its usual topology) is Hausdorff. That is, for any two different points $v'$ and $v''$, both in $conv\left(C_N\right)$, there exist an open set $O_v'$ containing $v'$ and open set $O_{v''}$ containing $v''$, such that $O_v'$ and $O_{v''}$ are disjoint.
\end{lem}
\begin{proof}
Any two points in $conv(C_N)$ are points in $\Re^N$, which can be separated from each other by open sets in $\Re^N$. The parts of these open sets that are in $conv(C_N)$ are open in $conv(C_N)$, and they serve as the open sets separating the two points.
\end{proof}

\begin{lem}
Let $N$ be a positive integer. Let $\Re^{2M}$, $Nam_N$, $Nam_N/{\cong}$, $\Re^N$, and $conv\left(C_N\right)$ have their usual topologies. Then, $F_{QNam}$ is a homeomorphism from $Nam_N/{\cong}$ to $conv\left(C_N\right)$.
\end{lem}
\begin{proof}
Let $M$ be $\frac{N+1}{2}$ if $N$ is odd, or $\frac{N+2}{2}$ if $N$ is even.

By Lemma \ref{naminglem}, the Naming Lemma, every point 
$
\left(
\begin{matrix}
v_1 \\
\cdots \\
v_N \\
\end{matrix}
\right)
$ in the convex hull of $C_N$ has an $\left(\frac{N+1}{2}\right)$-naming. Let $P$ be this naming, which thus evaluates to 
$
\left(
\begin{matrix}
v_1 \\
\cdots \\
v_N \\
\end{matrix}
\right)
$ 
when treated as a convex combination. Thus, 
$
F_{Nam}\left(P\right) =
\left(
\begin{matrix}
v_1 \\
\cdots \\
v_N \\
\end{matrix}
\right)
$. 
$P$ is in some equivalence class $\tilde{P}$, for which, it thus holds that 
$
F_{QNam}\left(\tilde{P}\right) =
\left(
\begin{matrix}
v_1 \\
\cdots \\
v_N \\
\end{matrix}
\right)
$. That makes $F_{QNam}$ onto.

If $F_{QNam}\left(\tilde{P_1}\right) = F_{QNam}\left(\tilde{P_2}\right)$ for two equivalence classes of $\left(\frac{N+1}{2}\right)$-namings, then let $P_1 \in \tilde{P_1}$ and $P_2 \in \tilde{P_2}$. Thus, $F_{Nam}\left(P_1\right) = F_{Nam}\left(P_2\right)$, so $P_1$ and $P_2$ evaluate to the same point. Thus, $P_1$ and $P_2$ are equivalent to each other, by Lemma \ref{equivalencelemma}, the Equivalence Lemma. Thus, $\tilde{P_1}$ and $\tilde{P_2}$ are the same. That makes $F_{QNam}$ one-to-one.

By Lemma \ref{continuitylemma}, the Continuity Lemma, $F_{QNam}$ is continuous.

By Lemma \ref{compactnessofnamnovercong}, the Compactness of $Nam_N/{\cong}$ Lemma, $Nam_N/{\cong}$ is compact.

By Lemma \ref{hausdorffconvcnlemma}, the Hausdorff Property of $conv\left(C_N\right)$ Lemma, $conv\left(C_N\right)$ is Hausdorff.

Thus, $F_{QNam}$ is an onto, one-to-one, continuous function from a compact space ($Nam_N/{\cong}$) to a Hausdorff space ($conv\left(C_N\right)$). Thus, $F_{QNam}$ is a homeomorphism from $Nam_N/{\cong}$ to $conv\left(C_N\right)$.
\end{proof}

\begin{rmk}
When the namings in the equivalence classes are considered as convex combinations, $F_{QNam}$ is exactly the function that evaluates the convex combination.
\end{rmk}

\section{Conclusion}

This paper examines an aspect of the curve $\left(t, t^2, t^3, \cdots, t^N\right)$, namely its convex hull when $t$ is restricted to a closed interval. Not only do $\frac{N+1}{2}$ points on the curve suffice to name each point in this convex hull as a convex combination, but, up to equivalence, there is only one such way to name each point in the convex hull, and also, the evaluation of the convex combination as a sum is a homeomorphism.

\bibliographystyle{plain}

\begin{thebibliography}{10}

\bibitem{doi:10.1137/090746525}
João Gouveia, Pablo~A. Parrilo, and Rekha~R. Thomas.
\newblock Theta bodies for polynomial ideals.
\newblock {\em SIAM Journal on Optimization}, 20(4):2097--2118, 2010.

\bibitem{Henrion2011}
Didier Henrion.
\newblock Semidefinite representation of convex hulls of rational varieties.
\newblock {\em Acta Applicandae Mathematicae}, 115(3):319, 2011.

\bibitem{Part1}
Kostyantyn Mazur.
\newblock A partial solution to continuous {B}lotto.
\newblock 2017.

\bibitem{Ranestad2011}
Kristian Ranestad and Bernd Sturmfels.
\newblock {\em The Convex Hull of a Variety}, pages 331--344.
\newblock Springer Basel, Basel, 2011.

\bibitem{Ranestad-Sturmfels}
Kristian Ranestad and Bernd Sturmfels.
\newblock On the convex hull of a space curve.
\newblock {\em Advances in Geometry}, 12(1):157--178, March 2012.

\bibitem{SCHEIDERER20112606}
Claus Scheiderer.
\newblock Convex hulls of curves of genus one.
\newblock {\em Advances in Mathematics}, 228(5):2606 -- 2622, 2011.

\bibitem{Scheiderer-Semidefinite}
Claus Scheiderer.
\newblock Semidefinite representation for convex hulls of real algebraic
  curves.
\newblock October 2012.
\newblock arXiv:1208.3865v3 on arXiv.

\bibitem{Sedykh1977}
V.~D. Sedykh.
\newblock Singularities of the convex hull of a curve in {$\Re^3$}.
\newblock {\em Functional Analysis and Its Applications}, 11(1):72--73, 1977.

\bibitem{Sinn2015}
Rainer Sinn.
\newblock Algebraic boundaries of convex semi-algebraic sets.
\newblock {\em Research in the Mathematical Sciences}, 2(1):3, 2015.

\bibitem{Schweighofer-Sturmfels-Thomas}
Markus Schweighofer{,}~Bernd Sturmfels{,} and Rekha Thomas, editors.
\newblock {\em Convex algebraic geometry}, February 2010.
\newblock This is the final report of a workshop.

\bibitem{Vinzant}
Cynthia Vinzant.
\newblock {\em Real Algebraic Geometry in Convex Optimization}.
\newblock {Ph.D.} dissertation, University of California, Berkeley, Spring
  2011.

\bibitem{WeissteinImplicit}
{Weisstein, Eric W.}
\newblock ``{Implicit Function Theorem.}'' {From MathWorld---A Wolfram Web
  Resource}.
\newblock http://mathworld.wolfram.com/ImplicitFunctionTheorem.html. Last
  visited on 07-Oct-2016.

\end{thebibliography}

\end{document}